\documentclass[11pt,a4paper]{amsart}
\usepackage[utf8x]{inputenc}
\usepackage{amsmath}
\usepackage{amssymb}
\usepackage{graphicx}
\usepackage{color}
\usepackage{amsthm}
\usepackage{bbm}
\usepackage{hyperref}
\numberwithin{equation}{section}
\title[Erd\H{o}s-R\'{e}nyi law]{The Erd\H{o}s-R\'{e}nyi law of large numbers for ballistic random walk in random environment}
\author{Darcy Camargo}
\address{Department of Mathematics, Weizmann Institute, POB 26, Rehovot 7610001, Israel}
\email{darcygcamargo@gmail.com}
\author{Yuri Kifer}
\address{Department of Mathematics, Hebrew University,  Givat Ram,
Jerusalem 9190401,
Israel}
\email{yuri.kifer@mail.huji.ac.il}
\author{Ofer Zeitouni}
\address{Department of Mathematics, Weizmann Institute, POB 26, Rehovot 7610001, Israel}
\email{ofer.zeitouni@weizmann.ac.il}
\thanks{This project has received funding from the European Research Council (ERC) under the European Union's Horizon 2020 research and innovation programme (grant agreement No. 692452)}

\newcommand{\e}{\varepsilon}

\newcommand{\pr}{\mathbb{P}}
\newcommand{\qr}{\mathbb{Q}}
\newcommand{\btau}{\bar{\tau}}
\newcommand{\pw}{\pr^{\omega}_0}
\newcommand{\ex}{\mathbb{E}}
\newcommand{\ind}{\mathbbm{1}}

\newcommand\aslimit{\mathrel{\overset{\makebox[0pt]{\mbox{\tiny a.s.}}}{\rightarrow}}}

\DeclareMathOperator*{\argmax}{arg\,max}

\newtheorem{teo}{Theorem}[section]
\newtheorem{lemma}[teo]{Lemma}
\newtheorem{df}[teo]{Definition}
\newtheorem{prop}[teo]{Proposition}

\newtheorem{cond}[teo]{Condition}
\begin{document}
\date{April 18, 2020}
\bibliographystyle{plain}

\begin{abstract}
We consider a one dimensional ballistic nearest-neighbor random walk in a random environment. We prove 
an Erd\H{o}s-R\'enyi strong law for the increments. 
\end{abstract}
\maketitle 
\section{Definitions and main results}

The classical  Erd\H{o}s-R\'enyi strong law of large numbers asserts as follows.
\begin{teo}[Erd\H{o}s-R\'enyi, 1970] 
  \label{theo-ER} Consider a  random walk $S_n=\sum_{i=1}^n X_i$ with $X_i$ i.i.d., satisfying $EX_1=0$. Set  $\phi(t)=E[e^{tX}]$ and let 
${\mathcal D}_\phi^+=\{t> 0: \phi(t)<\infty\}$. 
 Let  $\alpha>0$ be
 such that $\phi(t)e^{-\alpha t}$ achieves its  minimum value for some  $t$ in the interior of  ${\mathcal D}_\phi^+$. Set $1/A_\alpha:=
-\log \min \limits_{t>0} \phi(t) e^{-\alpha t}$. Then, $A_\alpha>0$ and 
\begin{equation}
  \label{eq-ERstat}
\max \limits_{0\leq j\leq n-\lfloor A_\alpha \log n\rfloor} 
\frac{S_{j+\lfloor A_\alpha \log n\rfloor}-S_j}{\lfloor A_\alpha
\log n\rfloor}\aslimit \alpha, \quad a.s.
\end{equation}
\end{teo}
\noindent
In the particular case of $X_i\in \{-1,1\}$, the assumptions of the theorem 
are satisfied for any $\alpha\in (0,1)$. 
The theorem also trivially generalizes to $EX_1\neq 0$,
by considering $Y_i=X_i-EX_i$.

Theorem \ref{theo-ER} is closely related to the large deviation principle
for $S_n/n$ given by Cramer's theorem, see e.g. \cite{DZ} for background. 
Indeed, with $I(x)=\sup_{t} (tx-\log \phi(t))$ denoting the rate function,
one observes that $I(\alpha)=1/A_\alpha$ and that 
\begin{equation}
  \label{eq-extrabla}
\alpha=\inf\{x>0: I(x)>1/A_\alpha\}.
\end{equation}


In this paper, we prove an analogous statement for standard one dimensional random walk in random environment (RWRE), in the case of positive velocity. We begin by introducing the model.
Fix a realization $\omega=\{\omega_i\}_{i\in \mathbb{Z}}$ with $\omega_i\in (0,1)$ of a
collection of  i.i.d. random variables,
which we call the \textit{environment}. With $p$ denoting the law of 
$\omega_0$ and $\sigma(p)$ its support,  denote by
 $P=p^{\mathbb{Z}}$ the law of the environment on $\Sigma_p:=\sigma(p)^\mathbb{Z}$. We make throughout the following assumption.
\begin{cond}[Uniform Ellipticity]\label{elipse} 
There exists a $\kappa\in (0,1)$ such that $\sigma(p) \subset [\kappa,1-\kappa]$  almost surely.
\end{cond}
\noindent 
Letting $\rho_i:=(1-\omega_i)/\omega_i$, we note that the ellipticity assumption gives a deterministic uniform upper and lower bounds on $\rho_i$.

It will be useful for us to consider also 
different laws of the environment $\Sigma=[\kappa,1-\kappa]^{\mathbb Z}$, not necessarily product laws. Such laws will be denoted $\eta$.
Equipping $\Sigma$ with the standard shift, the space of   measures (stationary/ ergodic) on $\Sigma$ are  denoted  $M_1(\Sigma)$ ($M_1^s(\Sigma)$/$M_1^e(\Sigma)$), respectively; similar definitions hold
when $\Sigma$ is replaced by $\Sigma_p$.

On top of $\omega$ we consider the RWRE, which is a 
nearest neighbor random walk $\{X_t\}_{t\in \mathbb{Z}}$.  
Conditioned on the environment
$\omega$,  
$\{X_t\}$ is a Markov chain with transition probabilities 
\[
\pi(i,i+1)=1-\pi(i,i-1)=\omega_i.
\] 
We denote the law of the random walk, started at $i\in {\mathbb Z}$ and conditioned on a fixed realization of the environment $\omega$, by $\pr_i^\omega$  (the so-called \textit{quenched} law).
 For any measure
$\eta\in M_1(\Sigma)$,   the measure
$\eta(d\omega)\otimes \pr_i^\omega$ is referred to as the \textit{annealed} law, and denoted
by $\pr^{a,\eta}_i$; with some abuse of notation, we sometimes say annealed law for the restriction of $\pr^{a,\eta}_i$ to path space. If $\eta=P$ then we write
$\pr^a_i$ for $\pr^{a,P}_i$. We use similar conventions for expectations, e.g. $ \mathbb{E}^a_i$ 
for expectation with respect to $\pr^{a,P}_i$, etc.

\subsection{The potential $V$ and functional $S$}

Introduce the \textit{potential function}, which is defined as
\begin{equation}
    V_\omega(j)=\left\{\begin{array}{ll} \sum \limits_{i=1}^{j} \log \rho_i(\omega),& \mbox{ if }j>0;\\ 0,& \mbox{ if }j=0;\\   -\sum \limits_{i=j+1}^{0} \log \rho_i(\omega),& \mbox{ if }j<0.    \end{array} \right. 
\end{equation}
and the Lyapunov function, see \cite{popovcomets},
\begin{equation}
   \label{defS} 
   S(n, \omega)=\left\{\begin{array}{ll} \sum \limits_{i=0}^{n-1} e^{V_\omega(i)},& \mbox{ if }n>0,\\ 0,& \mbox{ if }n=0,\\   \sum \limits_{i=n}^{-1} e^{V_\omega(i)},& \mbox{ if }n<0.    \end{array} \right. 
\end{equation}
By definition,  for $n>m\geq 0$ we can decompose $S(n,\omega)$ as 
\begin{equation}
   \label{Sdecomp} S(n, \omega)=S(m,\omega)+e^{V_\omega(m)}S(n-m,\theta^{m}\omega).
\end{equation}
Another important property of $S(n,\omega)$ is its relation to hitting times. Let $\tau_A=\inf\{t>0:X_t\in A \}$ and abbreviate  $\tau_{\{i\}}=\tau_i$ for $i\in {\mathbb Z}$. Then,
see e.g. \cite[(2.1.4)]{coursezeitouni}, 
\begin{align}
\label{lemmaS} \pr^{\omega}_x[\tau_0>\tau_{y}]=\frac{S(x,\omega)}{S(y,\omega)},\quad\, \mbox{\rm for} \; y>x>0 .
\end{align}
Also, for $n>0$,
\begin{equation}
\label{eqS}e^{\max \limits_{0\leq j\leq n-1} V_{\omega}(j) }\leq S(n,\omega) \leq n e^{\max \limits_{0\leq j\leq n-1} V_{\omega}(j)}.
\end{equation}


\subsection{Rate functions and modified  environments}
 We follow \cite{LDP} in introducing the function
\[
\phi (\omega, \lambda)=\ex_0^{\omega} [e^{\lambda\tau_1}\ind[\tau_1<\infty]],
\]
and the hitting  time  quenched rate function, defined for $\eta\in M_1(\Sigma)$,
\begin{equation}
  \label{eq-Iqdef}
I^{\tau,q}_\eta(u)=\sup \limits_{\lambda\in \mathbb{R}} \big\{\lambda u - \int \log \phi(\omega, \lambda)\eta(d\omega) \big\}.
\end{equation}
We denote the empirical field $R_n\in M_1(\Sigma)$ by
\begin{equation}
  \label{eq-Rn}
R_n= \frac{1}{n}\sum \limits_{j=0}^{n-1}\delta_{\theta^j\omega}.
\end{equation}
It is well known, see e.g. \cite{DZ}, that under $P$,
the sequence $R_n$ satisfies a large deviation principle in $M_1(\Sigma)$, equipped with the topology of weak convergence, with rate function $h(\cdot|\eta)$, the so-called specific relative entropy.

 
We need to consider the RWRE  conditioned on  not hitting the origin, i.e. conditioned on
 $\tau_0=\infty$. Using Doob's h-transform, it is straightforward to check that such conditional law is equivalent to using a transformed environment, namely for all measurable $A$ and $i\in {\mathbb Z}_+$,
 \[
 \pr_i^{\omega}[A\mid \tau_0=\infty]=\pr_i^{\hat{\omega}}[A], 
 \]
 where
 \begin{equation}
  \label{eq-hatomega}
 \hat{\omega}_i=\frac{\omega_i S(i+1,\omega)}{S(i,\omega)}.
 \end{equation}
 Note that due to \eqref{defS}, we have that  $\hat{\omega}_i\in [0,1]$.
 
For $L$ a positive integer, 
consider the following ergodic (with respect to shifts, 
if the law of $\omega$ is ergodic)
environment obtained as a transformation of $\omega$,
 \begin{equation}
   \label{eq-TL}
\hat\omega^L_i:=\frac{\omega_iS(L+1, \theta^{i-L}\omega)}{S(L, \theta^{i-L}\omega)}.
\end{equation}
Here again, $\hat\omega^L_i\in [0,1]$.
Introduce the function
\begin{equation}
  \label{eq-IL}
  I^F(x,\eta)= \lim \limits_{L\to \infty}\sup \limits_{\lambda\leq 0} 
  \big\{\lambda  -x \int \log \phi(\hat\omega^L, \lambda)\eta(d\omega) \big\},
\quad \eta\in M^s_1(\Sigma).
\end{equation}
The existence of the limit in \eqref{eq-IL}
is due to the following lemma, whose proof appears in Section \ref{sec-lem13}.
 \begin{lemma} \label{lemma2}For any fixed $i$, the sequence $\{\hat\omega_i^L \}$ is decreasing in $L\in \mathbb{Z}^+$. Moreover the limit in \eqref{eq-IL} exists for any $\eta\in M^s_1(\Sigma)$.
 \end{lemma} 
For $\eta\in M_1^e(\Sigma)$, $I^F(x,\eta)$ has a natural interpretation as a rate function for the quenched LDP of the hitting times of the random walk in random environment,  conditioned on never hitting the origin, see Appendix~A.
 
 \subsection{Statement of main result}
 
 With all needed information gathered, we state the main result of the paper.

\begin{teo}\label{mainteo} 
Let $P=p^{\mathbb{Z}}$ satisfy Condition~\ref{elipse}. Set
\begin{equation}
  \label{def-s}
  s=\sup\{\theta>0:E_p \rho_0^\theta  \leq 1\}.
\end{equation}
Assume that  $s\in (1,\infty]$.
Fix $A>0$.  Then, for $k=k(n)$ positive integer such that $k(n)/\log n\to A$,
\begin{equation}
  \label{ER-RWRE}
    \max \limits_{1\leq t\leq n-k} \frac{X_{t+k}-X_t}{k}
    \to_{n\to\infty} x^*(A), \quad \pr^{a}_0-a.s.,
\end{equation}
where
\begin{equation}
  \label{eq-IRWRE}
x^*(A)=\inf\{x>0:I^{*}(x)>1/A \},
\end{equation}
and
\begin{equation}
  \label{eq-Is}
I^*(x)=\inf \limits_{\eta \in M_1^s(\Sigma_p)}\Big\{I^{F}\big(x,\eta\big)+xh(\eta|P)\Big\}.
\end{equation}
\end{teo}
\noindent (Compare \eqref{ER-RWRE} and \eqref{eq-IRWRE} with
\eqref{eq-ERstat} and \eqref{eq-extrabla}.)

Let 
\begin{equation}
  \label{eq-vp}
  v_p:=\frac{1-E_p(\rho_0)}{1+E_p(\rho_0)}.
\end{equation}
We remark, see \cite{coursezeitouni}, that the condition $s\in (1,\infty]$
is equivalent to $E_p(\rho_0)<1$ and 
is also equivalent to
the convergence 
\begin{equation}
  \label{eq-LLNconv}
  \frac{X_n}{n} \to_{n\to\infty}  v_p>0, \quad  \pr^a-a.s..
\end{equation}
That is, we are dealing here with the transient ballistic case. It also implies that $E_p \log \rho_0 <0$.

We further note that
it follows from the definitions that $x\mapsto I^*(x)$ is a convex increasing 
function on $\mathbb{R}_+$, with $I^*(0)=0$ and $I^*(x)\to_{x\to\infty} \infty$.
Thus, $I^*$ is continuous on its domain  and
strictly increasing in 
the set $\{x: \infty>I^*(x)>0\}$.
Therefore, $x^*(A)$ is well defined and satisfies $AI^*(x^*(A))=1$.
It is also obvious from Theorem \ref{mainteo} that $x^*(A)\leq 1$.

\subsection{Proof strategy}
The standard proof of Theorem \ref{theo-ER} and of its 
extensions to sums of weakly dependent random variables 
usually consists of an upper
and of a lower bounds for increments within time intervals (which we refer to
as  temporal blocks) of length $A_\alpha \log n$. 
The former relies only on the upper 
large deviations bound for such sums 
 while the latter in addition to the lower
large deviations bound requires also sufficiently weak dependence which enables to split the sum into weakly dependent disjoint blocks
(this step is, of course, trivial in the independent case). In this way the corresponding random walk is split into weakly dependent temporal blocks. Such a temporal splitting  is not possible 
in our case of random environment, since (under the annealed measure) 
increments of the random walk in disjoint time intervals are strongly  
correlated. So instead, in the proof of Theorem \ref{mainteo}
we use a spatial decoupling of the walk in order to obtain both
upper and lower bounds on maximal increments. 
This leads to several complications. 
First, the increments of the walk in different spatial blocks are not independent. Secondly, and more important, the walk may visit a block many times, and the probability to do so depends not only on the environment in the block but also
on adjacent blocks.

The first difficulty is relatively easily dispensed with by appealing to a standard 
non-backtracking estimate (Lemma \ref{LemmaApendix}). 
This allows us to consider only blocks of size $c\log n$ for some large $c$.
To address the second issue, we use the environment $\hat \omega$,
see \eqref{eq-hatomega}, representing the environment under the condition
of not backtracking at all, and use it to introduce the crucial
quantity $\chi(k,x,c,\eta)$ which serves as a proxy for the probability 
of having a fast segment of the walk in a block of length $xk$ with $k=k(n)$ such that $k/\log n \to A$, under the ergodic measure $\eta$, see \eqref{eq-chi}
for the definition 
and the crucial Lemma \ref{mainlemma} for the representation of the 
maximal increment in terms of $\chi$ (by fast segment we mean a segment which
crosses the block faster than typical, that is with speed $1/x$). The rest of the proof
involves a study of $\chi$, which is an expectation
(with respect to $P$) of functions of the environment (some of which represent 
quenched large deviations). As in \cite{LDP},
the latter can be represented in terms of a variational problem involving a
change in the environment, and a function of quenched large deviations estimate
for the RWRE, see \eqref{eq-Is} for the form 
of the variational principle. 

We remark that the proof of Lemma \ref{lemma2} requires several 
approximation steps due to the fact that the environemnt $\hat \omega$
is not an ergodic environment under $\eta$. This is carried out in Section \ref{sec-lem13}.
On the other hand, the identification of the rate function requires 
a study of the variational principle, and it is in 
the latter study that we use the assumption that $s>1$, see
the statement of Theorem \ref{mainteo}.

\subsection{Notation}
For two sequences (of possibly random variables)
$a_n$ and $b_n$ we will say $a_n\sim b_n$ if it holds almost surely that 
\[
\lim_{n \to \infty} \frac{a_n}{b_n}=1.
\]
 We say that $a_n=o(b_n)$ if $a_n/b_n \to 0$ almost surely (with respect to
 the measure under consideration) as $n\to \infty$.
Constants, whose values are fixed throughout the paper, are denoted
by $\alpha_i$ and we fix 
\begin{equation}
  \label{eq-ckappa}
  C_\kappa := (1-\kappa)/\kappa>1. 
\end{equation}

The shift operator on $\Sigma$ is denoted by $\theta$, so $(\theta^i\omega)_j:=\omega_{i+j}$.
 We also define the flipped and reversed 
 environments $\bar{\omega}$ and $r(\omega)$ by
 \begin{equation}
   \label{eq-mrs2}
   \bar{\omega}_i:=1-\omega_i, \qquad 
   r(\omega)_i:=\omega_{-i}.
 \end{equation}

Recall that  $\tau_i=\inf\{t>0: X_t=i\}$. We denote the subsequent visits to a site by $\tau_i(j)=\inf \{t>\tau_i(j-1): X_t=i \}$, $j\geq 2$, with $\tau_i(1)=\tau_i$.
We denote by $\ell(i,t)$ the partial local time of a 
site $i$ up to time $t$, i.e.
\begin{equation}
  \label{eq-localtime}
\ell(i,t):=\sum \limits_{j=0}^{t}\ind[X_j=i].
\end{equation}
The  local time at $i$ is defined as $\ell(i):=\lim_{t\to \infty} \ell(i,t)$.

We also define some functionals that depend on $V$ and $S$ and will be useful later.
\begin{align}
    \label{defW} W(n,\omega)=\frac{e^{V_\omega(n)}}{S(n,\omega)},
\end{align}
\begin{align}
    \label{defxi} \xi_n(i,\omega)=\frac{W(i+1,\theta^{n-i}\omega)}{1+S(i,\theta^{n-i}\omega)W(n-i,\omega)},
\end{align}
and
\begin{align}
    \label{defxibar} \bar\xi(i,\omega)=\frac{W(i+1,\omega)}{1+{S(i,\omega)}/{S(-\infty,\omega)}}.
\end{align}
We will see in \eqref{Winverse} below that
$\xi_n(i,\omega)\leq \bar\xi(i,\theta^{n-i}\omega)$.



\section{Non-backtracking estimate}
We provide in this section non-backtracking estimates which will be crucial in
obtaining spatial decoupling of events.
\begin{lemma}\label{Wbound} Assume 
  $\eta\in M_1(\Sigma)$. Then, for every $n\geq 1$,
  \begin{equation}\label{eq-mrs1}
W(n,\omega)\leq \frac{1-2\kappa}{\kappa}\Big(1-\Big(\frac{\kappa}{1-\kappa}\Big)^n \Big)^{-1} \leq  \frac{1-\kappa}{\kappa}, 
\quad \eta-\mbox{a.s.}.
\end{equation}
\end{lemma}
\begin{proof}
First observe that $W(n,\omega)$ satisfies 
\begin{align*}
  \frac{1}{W(n,\omega)}=\frac{1}{\rho_{n}}\Big(1+\frac{1}{W(n-1,\omega)}, \Big)
\end{align*}
and hence 
\begin{align*}
    \frac{1}{W(n,\omega)}\geq \frac{\kappa}{1-\kappa}\Big(1+\frac{1}{W(n-1,\omega)}\Big).
\end{align*}
Inductively applying this relation we get 
\begin{align*}
    \frac{1}{W(n,\omega)}\geq \sum \limits_{j=1}^{n-i}\Big( \frac{\kappa}{1-\kappa}\Big)^{j} +\Big(\frac{\kappa}{1-\kappa}\Big)^{n-i}\frac{1}{W(i,\omega)}.
\end{align*}
Using that $W(1,\omega)=\rho_0\leq (1-\kappa)/\kappa$ we conclude
\begin{align*}
    \frac{1}{W(n,\omega)}\geq \sum \limits_{j=1}^{n}\Big( \frac{\kappa}{1-\kappa}\Big)^{j} =\Big( \frac{\kappa}{1-\kappa}\Big)\frac{1-\Big(\frac{\kappa}{1-\kappa}\Big)^n }{1-\Big( \frac{\kappa}{1-\kappa}\Big)},
\end{align*}
yielding the first inequality in \eqref{eq-mrs1}. The second inequality follows from monotonicity in $n$.
\end{proof}
For $a$ a positive integer, set  $\tau^{\text{BT}}_y(a)=\inf\{t>\tau_y:X_t=y-a\}$ (possibly $\tau^{\text{BT}}=\infty$)
as the first backtracking time of $a$ steps for the walk after hitting $y$, and introduce the backtracking event
\begin{align}
    \label{BTevent2} B(n,a)=\bigcup \limits_{y=1}^{n} \{\tau^{\text{BT}}_y(a)<\tau_{y+1}\}.
\end{align}
 The following standard lemma shows that large logarithmic in $n$ 
backtrackings
are not occuring before hitting position $n$.
\begin{lemma}
\label{LemmaApendix} 
  Assume that $P=p^{\mathbb{Z}}$ satisfies the conditions of
  Theorem \ref{mainteo}.
   Then there exists a constant 
  $\alpha_0>0$ so that, for any
$A>0, c>0$ and $k=k(n)\sim A\log n$ so that $ck$ is an integer, 
and all $n$ large,
\begin{equation}\label{apendix2}
    \pr^{a}_0[B(n,ck)]\leq  n^{1-\alpha_0cA}.
\end{equation}
\end{lemma}

\begin{proof}
 Observe that 
\[\pw[B(n,ck)]=\pw\Big[\bigcup 
\limits_{y=1}^{n} \{\tau^{\text{BT}}_y(ck)<\tau_{y+1} \} \Big]
\leq \sum \limits_{y=1}^{n} \pw[\tau^{\text{BT}}_y(ck)<\tau_{y+1} ],
\]
and therefore,
\begin{align}
  \label{intS}
     \pr^a_0[B(n,ck)] &\leq n\pr^{a}_{ck}[\tau_{0}<\tau_{ck+1} ]
    =n \int \Big(1-\frac{S(ck,\omega)}{S(ck+1,\omega)}\Big) P(d\omega)\\
    \nonumber &=n \int \Big(\frac{e^{V_\omega(ck)}}{S(ck+1,\omega)}\Big) P(d\omega)
    =n\int \Big(\frac{W(ck,\omega)}{W(ck,\omega)+1}\Big) P(d\omega).
\end{align}
Let $\mu=\int \log \rho_0 dp$, which is negative by assumption. 
From \eqref{intS} and
$$\frac{W(ck,\omega)}{W(ck,\omega)+1}\leq \min(1,e^{V_\omega(ck)})$$
we obtain that, for all large $n$, 
\begin{align*}
  \pr^a_0[B(n,ck)]
    &\leq n\int e^{V_\omega(ck)}\ind[V_\omega(ck)\leq ck\mu/2 ] P(d\omega)+nP(V_\omega(ck)> ck\mu/2)\\
    &\leq  n(e^{ck\mu/2}+e^{-\alpha ck})\leq e^{-\alpha_{0}ck}
\end{align*}
with $\alpha,\alpha_0$ depending on $p$ only, where the 
second inequality is due to Hoefding's inequallity.
Recalling that $k\sim A\log n$ concludes the proof.
\end{proof}
\section{A reduction to block estimates, large deviations, and proof of Theorem \ref{mainteo}}
  \label{sec-red}
  In this section we reduce the Erd\"{o}s-Renyi problem to a block estimate, and state a large deviations estimate for the block. The proof of both these estimates  
  is technical and will be provided in subsequent sections. We then show how the
  block estimates yield the proof of Theorem \ref{mainteo}.

  For $k$ integer and $c,x\in \mathbb{R}_+$, set
\begin{equation}
\label{fgdef}f(\omega,x,k)=\pr^{\hat{\omega}}_1[ \tau_{xk}> k]\quad
\mbox{ and } \quad g(\omega,x,c,k)=S(xk,\omega)/S(ck,\omega),
\end{equation}
where $\hat\omega$ is as in \eqref{eq-hatomega}. (Here and in the sequel,
we abuse notation by writing $xk$ and $ck$ instead of $\lfloor xk\rfloor$
or $\lfloor ck\rfloor$, as appropriate.)
For $\eta\in M_1^s(\Sigma)$, set
\begin{align}
\label{eq-chi}
\chi  (k,x,c,\eta)&:=\int \Big( \frac{1-f(\omega,x,k)}{1-f(\omega,x,k)(1-g(\omega,x,c,k))}\Big)\eta(d\omega).
\end{align}
When $\eta=P$, we omit $\eta$ from the notation and write $\chi(k,x,c)$ for $\chi  (k,x,c,\eta)$. 


  \subsection{A block estimate}
  Introduce the notation
\begin{equation}
\label{eq-maxx}
\dot{ X}(k,n):=\max \limits_{1\leq t\leq n-k}\frac{X_{t+k}-X_{t}}{k}.
\end{equation}
  The main result of this subsection, whose proof is postponed to
  Section \ref{sec-block}, is the following lemma. Recall the asymptotic velocity
  $v_p$, see \eqref{eq-vp}.
\begin{lemma}\label{mainlemma} 
  Fix $A>0$ and set
   $k=k(n)\sim A\log n$ integer. Then, there exists a constant
   $\alpha_1>0$ so that for any $c>x>0$, 
   and all $n$ large enough,
%
\begin{equation}\label{Pbound1}\pr_{0}^{a}\Big[\dot{X}(k,\tau_{\lfloor v_pn\rfloor})\geq x\Big]\leq n\chi  (k,x,c)+n^{1-\alpha_1Ac},
\end{equation}
and 
\begin{equation}\label{Pbound2}
\pr_{0}^{a}\Big[\dot{X}(k,\tau_{v_pn})
< x\Big]\leq \exp \Big(-\big(\lfloor\frac{v_p n-k}{ck}\rfloor-1\big) 
\chi  (k,x,c)\Big)+n^{1-\alpha_1Ac}.
\end{equation}
\end{lemma}
\noindent Note that the statement of the lemma is trivial if $x>1$, for then 
$f(\omega,x,k)=1$ and $\chi(k,x,c,\eta)=0$, as  expected.

In the rest of the paper, we always take $c>c_0$ where $\alpha_iAc_0>2$, $i=0,1$. This ensures that the error terms in the right hand side of \eqref{Pbound1} and \eqref{Pbound2}, and also of \eqref{apendix2}, are summable.
We also recall our convention to write for brevity 
$ck$, $xk$ instead of $\lfloor ck\rfloor$,
$\lfloor xk\rfloor$. 
\subsection{A logarithmic estimate for $\chi$}
The following  result, which gives a representation 
of $\chi$ from Lemma \ref{mainlemma}
in terms of the function $I^*$,
is a crucial ingredient in the proof of
Theorem \ref{mainteo}. Its proof is technically involved and postponed to Section
\ref{sec-varprob}.
\begin{prop}\label{limitteo}
  Under the assumptions of Theorem \ref{mainteo}, and $c>c_0$, 
  we have that
\[
\lim \limits_{k\to \infty} -\frac{1}{k}\log\chi  (k,x,c)=I^*(x),
\]
where $I^*$ is as in \eqref{eq-Is}.
\end{prop}
\subsection{Proof of Theorem \ref{mainteo}}
We now combine Lemma \ref{mainlemma} 
with Proposition~\ref{limitteo} to prove Theorem~\ref{mainteo}. 
Throughout, $k=k(n)$ is as in the statement of the theorem.
\begin{proof}[Proof of Theorem~\ref{mainteo}]
  Fix $\gamma>0$ small. 
  Let $E_0=E_0(\gamma)=\{n\leq \tau_{\lfloor v_p n/(1-\gamma)\rfloor}\}$ and $E_1=E_1(\gamma)=\{n\geq \tau_{\lfloor v_p n/(1+\gamma)\rfloor}\}$. 
By \eqref{eq-LLNconv},
  \begin{equation}
    \label{taulim} \frac{\tau_{\lfloor v_p n\rfloor }}{n} \to 1, \quad \pr^a_0-a.s.,
  \end{equation}
  implying that $E_0,E_1$ occur for all large enough $n$, almost surely under $\pr_0^a$.
From Lemma \ref{mainlemma} (applied with $n/(1-\gamma)$ and $n/(1+\gamma)$)
and the fact that  $k(n)\sim k(n/(1-\gamma))\sim k(n/(1+\gamma))$ we get the following bounds
\[
\pr_{0}^{a}\Big[E_0, \dot{X}(k,n)\geq  x\Big]\leq 
\frac{n}{1-\gamma}\chi  (k,x,c) + n^{1-\alpha_1 Ac},
\]
and 
\[
\pr_{0}^{a}\Big[E_1, \dot{X}(k,n)<x  \Big]\leq \exp \Big(-\big(\lfloor\frac{v_p n/(1+\gamma)-k}{ck}\rfloor-1\big) 
\chi  (k,x,c)\Big) + n^{1-\alpha_1 Ac}.
\]
From Proposition~\ref{limitteo} we obtain that
\begin{equation*}
    \lim \limits_{n\to \infty} \frac{\log \big(n\chi  (k,x,c)\big)}{\log n}= 1-AI^*(x),
\end{equation*}
and therefore
for every $\e>0$ there is a constant $\alpha_{\e}>0$ such that 
\begin{equation}\label{nchiapprox}
  \alpha_{\e}^{-1}n^{1-AI^*(x)-\e} \leq n\chi  (k,x,c)\leq \alpha_{\e}n^{1-AI^*(x)+\e},
\end{equation}
hence, for some constant $c'_\e>0$, 
\begin{align} \label{lowerb}
    \pr_{0}^{a}\Big[E_1,\dot{X}(k,n)<x  \Big]\leq \exp \Big(-\frac{c'_\e}{(1+\gamma)ck}n^{1-AI^*(x)-\e}\Big) + n^{1-\alpha_1 Ac}.
\end{align}
Fix now $x<x^*(A)$ and set
$\e=A(I^*(x^*(A))-I^*(x))/2$. Because $I^*$ is strictly increasing in 
a neighborhood of
$x^*(A)$, we have with $AI^*(x^*(A))=1$ that
$\e>0$ and $1-AI^*(x)-\e>0$, which together with
the 
choice
$c>c_0$, imply that  the right hand side of \eqref{lowerb} is summable. 
Together with
\eqref{taulim}, it follows from the Borel-Cantelli lemma that 
 $\liminf_{n\to\infty} \dot{X}(k,n) \geq x^*(A)$, $\pr^a$-a.s.

For the other bound let $n_j=\max\{n: k(n)=j \}$. Since $k\sim A\log n$, there
exist constants $\alpha_9,\alpha_{10}>0$ such that
for all values of $n$ with $k(n)=j$, we have that
$n\geq \alpha_{9}e^{\alpha_{10}j}$. 
Therefore by \eqref{nchiapprox} and \eqref{Pbound1}, 
we have for any $x> x^*(A)$ and $\e>0$ that
  \begin{align}
      \nonumber\sum \limits_{j=1}^{\infty} \pr_{0}^{a}\Big[E_0,\dot{X}(j,n_j)\geq x  \Big]&\leq \sum \limits_{j=1}^{\infty} \alpha_{\e} n_j^{1-AI^*(x)+\e}+C\\
      \label{assum1}&\leq \sum \limits_{j=1}^{\infty} c''_{\e} e^{\alpha_{10}j(1-AI^*(x)+\e)}+C,
  \end{align}
  where $C$ is some constant coming from the summation of the error term
  in \eqref{Pbound1}.
Taking $\e=A(I^*(x)-I^*(x^*(A)))/2$ and using that $I^*(x)>I^*(x^*(A))=1/A$
makes the sum in \eqref{assum1} finite, therefore
by the Borel-Cantelli lemma for all but a finite number of $j$'s we have that
$\dot{X}(k,n_j)\leq x$, almost surely. For every $n$ there is a $j$ such that $n\leq n_j$ and $k(n)=k(n_j)$, therefore
$\dot{X}(k,n)\leq \dot{X}(j,n_j)$, and hence for all but a finite number of $n$ we have $\dot{X}(k,n)\leq x$. 
Since  $x>x^*(A)$ is arbitrary, we obtain together with
\eqref{taulim} that
$$\limsup_{n\to\infty} \dot{X}(k,n)\leq x^*(A), \quad \pr^a-a.s.,$$
concluding the proof.
\end{proof}
\section{Proof of Lemma \ref{mainlemma}}
\label{sec-block}
We provide in this section the proof of Lemma \ref{mainlemma},
which was used in
the proof of Theorem \ref{mainteo}. 
\begin{proof}[Proof of Lemma \ref{mainlemma}]
 Fix $x,c$ as in the lemma.  It follows from
 Lemma~\ref{LemmaApendix} that
 \begin{equation}
   \label{eq-noBT}
   \pr^a_0[B(n,ck)]\leq n^{1-\alpha_0Ac}.
 \end{equation}

 We now begin the proof of 
 \eqref{Pbound1}. Note that on the complement of $B(n,ck)$,
no backtracking of length $ck$ occurs before the RWRE hits $n$.
For brevity, we write in the proof $v_p n$ for $\lfloor v_p n\rfloor$.
We bound 
\begin{align}
\nonumber
\pr_{0}^{\omega}&\Big[\dot{X}(k,\tau_{v_pn})
\geq x\Big] = \pr_{0}^{\omega}\Big[\bigcup\limits_{t=1}^{{\tau}_{v_p n}-k}\{X_{t+k}-X_{t}\geq xk\}\Big]\\
&\leq \pr_{0}^{\omega}\Big[B(v_p n,ck)^{\complement}\bigcap \big\{\bigcup\limits_{y=-c k}^{v_pn-xk}\bigcup\limits_{t=1}^{{\tau}_{v_p n}-k}\{X_{t+k}-y\geq xk, X_{t}=y\}\big\}\Big]\nonumber\\
&\qquad +\pw[B(v_p n,ck)].\label{eq-1ofer}
\end{align}
Turning to the first term in the right hand side of \eqref{eq-1ofer}, 
recalling the local time $\ell(\cdot,\cdot)$, see \eqref{eq-localtime}, we have
that
\begin{align}\nonumber
  &\pr_{0}^{\omega}\Big[B(v_p n,ck)^{\complement}\bigcap \big\{\bigcup\limits_{y=-c k}^{v_pn-xk}\bigcup\limits_{t=1}^{{\tau}_{v_p n}-k}\{X_{t+k}-y\geq xk, X_{t}=y\}\big\}\Big]\\
\nonumber &\leq\sum\limits_{y=-ck}^{v_pn-xk}\pr_{0}^{\omega}\Big[
B(v_p n,ck)^{\complement}\bigcap \big\{\bigcup\limits_{t=1}^{{\tau}_{v_p n}-k}\{X_{t+k}-y\geq xk, X_{t}=y\}\big\}\Big] \\
\nonumber&\leq\sum\limits_{y=-ck}^{v_pn-xk}\pr_{0}^{\omega}\Big[B(v_p n,ck)^{\complement}
\bigcap\big\{\bigcup\limits_{j=1}^{\ell(y,\tau_{v_p n}-k)}\{X_{\tau_y(j)+k}-y\geq xk\}\big\}\Big]\\
\label{uniprob}&=\sum\limits_{y=-ck}^{v_pn-xk}\pw[\tau_y<\tau_{v_p n}]\pr_{0}^{\theta^y\omega}\Big[B(v_p n,ck)^{\complement}\bigcap \big\{
\bigcup\limits_{j=1}^{\ell(0,\tau_{v_pn-y}-k)}\{X_{\tau_0(j)+k}\geq xk\}\big\}
\Big],
\end{align}
where $\{X_{\tau_0(j)+k}\geq xk\}=\varnothing$ if $\tau_0(j)=\infty$. 
Set $\bar{\tau}_y(0)=0$ and, for $j\geq 0$, define recursively
$\bar{\tau}_y(j)=\inf\{t>t_{y}(j): X_t=y\mbox{ or }X_t=y+xk\}$ and 
$t_y(j)=\inf\{t\geq\bar{\tau}_{y}(j-1): X_t=y\}$.
These are the consecutive attempts for the walk 
to cross the interval $[y,y+xk]$.
We represent 
the event $B(v_p n,ck)^{\complement}
  \bigcap \big\{\bigcup\limits_{j=1}^{\ell(0,\tau_{v_p n-y}-k)}\{X_{\tau_0(j)+k}\geq xk\}\big\}$ in the right hand side of \eqref{uniprob}  as
\begin{align*}
  &B(v_p n,ck)^{\complement}\bigcap \Big\{\bigcup\limits_{i=1}^{\infty}\{\btau_0(i)-t_0(i)\leq k,X_{\btau_0(i)}=xk\}\\
  &\qquad\bigcap_{j=1}^{i-1}\big[\{X_{\btau_0(j)}=0\}
  \bigcup \{\btau_0(j)-t_0(j)> k,X_{\btau_0(j)}=xk, t_0(j+1)<\tau_{v_p n-y}\}
\big]\Big\}
\end{align*} 
which is a subset of
\begin{align*}
  &B(v_p n,ck)^{\complement}
  \bigcap \Big\{\bigcup\limits_{i=1}^{\infty}\{\btau_0(i)-t_0(i)\leq k,X_{\btau_0(i)}=xk\}\\
  &\qquad\bigcap_{j=1}^{i-1}\big[\{X_{\btau_0(j)}=0\}
  \bigcup \{\btau_0(j)-t_0(j)> k,X_{\btau_0(j)}=xk, t_0(j+1)<\tau_{ck}\}\big]
\Big\}
\end{align*}
and therefore, using the Markov property,
\begin{align}
  \label{probDev}
\pr_{0}^{\omega}&\Big[B(v_p n,ck)^{\complement}
\bigcap \big\{\bigcup\limits_{j=1}^{\ell(0,\tau_{v_p n-y}-k)}\{X_{\tau_0(j)+k}\geq xk\}\big\}\Big]\\
\nonumber&\leq \sum\limits_{i=1}^{\infty}\pw [\btau_0(i)-t_0(i)\leq k,X_{\btau_0(i)}=xk]\\
\nonumber& \; \quad \times \big(\pw[\{X_{\btau_0(1)}=0\}\cup 
\{\btau_0(1)-t_0(1)> k,X_{\btau_0(1)}=xk, t_0(2)<\tau_{ck}\}]\big)^{i-1}\\
\nonumber&=\frac{\pw [\btau_0(1)\leq k,X_{\btau_0(1)}=xk]}{1-\pw[X_{\btau_0(1)}=0]-\pw[ \btau_0(1)> k,X_{\btau_0(1)}=xk, t_0(2)<\tau_{ck}]}\\
\nonumber
&=\frac{\pw [\btau_0(1)\leq k,X_{\btau_0(1)}=xk]}{1-\pw[X_{\btau_0(1)}=0]-\pw[ \btau_0(1)> k,X_{\btau_0(1)}=xk]\pr^{\omega}_{xk} [\tau_0<\tau_{ck}]}.
\end{align}
We calculate the  probabilities in the right hand side of \eqref{probDev}  
separately.
\begin{align*}
\pw[ \btau_0(1)> k,X_{\btau_0(1)}=xk]&=\pw[ \tau_{xk}> k,\tau_{0}>\tau_{xk}]\\
&=\omega_0 \pr^{\omega}_1[ \tau_{xk}> k-1,\tau_{0}>\tau_{xk}]\\
&=\omega_0\pr^{\omega}_1[ \tau_{xk}\geq k\mid \tau_{0}>\tau_{xk}]\pr^{\omega}_1[\tau_{0}>\tau_{xk}].
\end{align*}
Recall, see \eqref{eq-hatomega}, that ${\pr}^{\hat \omega}$ is  the law of the random walk 
in the environment $\omega$, conditioned on not hitting the origin. With this it holds by the Markov property
\begin{align*}
    {\pr}^{\hat\omega}_1[ \tau_{xk}\geq k]&={\pr}^{\omega}_1[ \tau_{xk}\geq k\mid \tau_0=\infty]=\lim \limits_{N\to \infty}{\pr}^{\omega}_1[ \tau_{xk}\geq k\mid \tau_0>\tau_N]\\
    &=\lim \limits_{N\to \infty}\frac{{\pr}^{\omega}_1[ \tau_{xk}\geq k\cap \tau_0>\tau_N]}{{\pr}^{\omega}_1[\tau_0>\tau_N]}\\
    &=\lim \limits_{N\to \infty}\frac{{\pr}^{\omega}_1[ \tau_0>\tau_{xk}\geq k] \pr^{\omega}_{xk}[\tau_0>\tau_N]}{{\pr}^{\omega}_1[\tau_0>\tau_{xk}]\pr^{\omega}_{xk}[\tau_0>\tau_N]}\\
    &=\pr^{\omega}_1[ \tau_{xk}\geq k\mid \tau_{0}>\tau_{xk}],
\end{align*}
and thus, 
\[
\pw[ \btau_0(1)> k,X_{\btau_0(1)}=xk]=\omega_0{\pr}^{\hat\omega}_1[ \tau_{xk}\geq k]\pr^{\omega}_1[\tau_{0}>\tau_{xk}].
\]
We also have
\begin{align}
\nonumber\pw[X_{\btau_0(1)}=0]&=\pw[\tau_0<\tau_{xk}]=(1-\omega_0)+\omega_0 \pr^{\omega}_1[\tau_0<\tau_{xk}]\\
&=(1-\omega_0)+\omega_0\Big(1-\frac{ 1}{S(xk,\omega)}\Big)
\label{p1frac}=1-\frac{\omega_0}{S(xk,\omega)},
\end{align}
Using \eqref{p1frac} in \eqref{probDev} we get
\begin{align}
\nonumber&\frac{\pw [\btau_0(1)\leq k,X_{\btau_0(1)}=xk]}{1-\pw[X_{\btau_0(1)}=0]-\pw[ \btau_0(1)> k,X_{\btau_0(1)}=xk]\pr^{\omega}_{xk} [\tau_0<\tau_{ck}]}\\
\nonumber&=\frac{\omega_0\pr^{\hat{\omega}}_1[ \tau_{xk}< k]\pr^{\omega}_1[\tau_{0}>\tau_{xk}]}{1-\pw[X_{\btau_0(1)}=0]-\omega_0\pr^{\hat{\omega}}_1[ \tau_{xk}\geq k]\pr^{\omega}_1[\tau_{0}>\tau_{xk}]\pr^{\omega}_{xk} [\tau_0<\tau_{ck}]}\\
\nonumber&\leq\frac{\omega_0\pr^{\hat{\omega}}_1[ \tau_{xk}< k](S(xk,\omega))^{-1}}{\omega_0(S(xk,\omega))^{-1}-\omega_0\pr^{\hat{\omega}}_1[ \tau_{xk}\geq k](S(xk,\omega))^{-1}\pr^{\omega}_{xk} [\tau_0<\tau_{ck}]}\\
\label{upperB1}&=\frac{\pr^{\hat{\omega}}_1[ \tau_{xk}< k]}{1-\pr^{\hat{\omega}}_1[ \tau_{xk}\geq k]\pr^{\omega}_{xk} [\tau_0<\tau_{ck}]}.
\end{align}

Using this in \eqref{eq-1ofer}--\eqref{probDev}
we get
\begin{align}
\nonumber\pr_{0}^{\omega}&\Big[\dot{X}(k,\tau_{v_pn})\geq x\Big]\\
\label{upperB2}&\leq \sum\limits_{y=-c k}^{v_pn-xk}\frac{\pr^{\widehat{\theta^y\omega}}_1[ \tau_{xk}< k]}{1-\pr^{\widehat{\theta^y\omega}}_1[ \tau_{xk}\geq k]\pr^{\theta^y\omega}_{xk} [\tau_0<\tau_{ck}]}+\pw[B(v_p n,ck)].
\end{align}

Taking expectations with respect to $P=p^{\mathbb Z}$ and using stationarity,
we obtain
\begin{align}
\nonumber \pr_{0}^{a}&\Big[\dot{X}(k,\tau_{v_pn})\geq x\Big]\\
\label{upperB2S}&\leq n \int \Big( \frac{1-f(\omega,x,k)}{1-f(\omega,x,k)(1-g(\omega,x,c,k))}\Big) P(d\omega)+\pr^a_0[B(v_p n,ck)].
\end{align}
Together with \eqref{eq-noBT},
this concludes the proof of \eqref{Pbound1}. 

We turn to the proof of \eqref{Pbound2}.
Recall that $k\sim A\log n$ and that we write $ck$ for $\lfloor ck\rfloor$.
Split the interval $[1, v_pn]$ into
blocks of size approximately $ck$. Let $z_{i} =  i ck$ for $0 \leq i \leq \lfloor(v_pn -  k)/ ck\rfloor$. Denote the
collection of visit times to the points $\{z_i\}$ by
$T = \{\tau_{z_{i}}(j)\}_{1\leq i\leq\lfloor (v_p n-k)/ ck\rfloor,j\in \mathbb{Z}+}$. We have
\begin{align}
\nonumber\pr_{0}^{\omega}\Big[\dot{X}(k,\tau_{v_pn})<x\Big] &=\pr_{0}^{\omega}\Big[\bigcap_{t=1}^{\tau_{v_p n}-k}\{X_{t+k}-X_{t}<xk\}\Big]\\
\nonumber&\leq \pr_{0}^{\omega}\Big[\bigcap_{t\in T\cap [0,\tau_{v_p n}-k]}\{X_{t+k}-X_{t}<xk\}\Big]\\
\label{doublecap}&=\pr_{0}^{\omega}\Big[\bigcap \limits_{i=1}^{\lfloor\frac{v_p n-k}{ ck }\rfloor}\bigcap \limits _{j=1}^{\ell(z_{i},\tau_{v_p n}-k)}\{X_{\tau_{z_{i}}(j)+k}-z_{i}<xk\}\Big].
\end{align}
The event $\{X_{\tau_{z_{i}}(j)+k}-z_{i}<xk\}$ only depends on the environment 
in the sites $[z_{i}-k,z_{i}+k]$. Even though we are considering disjoint blocks of the environment, there is still dependence on the local times. To induce the independence we will make use of another event, first  define $\tilde{\tau}_i=\inf\{t> \tau_{z_{i+1}}: X_t=z_{i} \}$ and
\[
\bar{B}_n(c,k)=\bigcap \limits_{i=1}^{\lfloor\frac{v_p n-k}{ ck }\rfloor-1} \{\tilde{\tau}_{i-1}>\tau_{z_{i+1}} \}.
\]
Then from \eqref{doublecap} and the Markov property we have
\begin{align}
\nonumber \pr_{0}^{\omega}&\Big[\bigcap \limits_{i=1}^{\lfloor\frac{v_p n-k}{ ck}\rfloor}\bigcap \limits _{j=1}^{\ell(z_{i},\tau_{v_p n}-k)}\{X_{\tau_{z_{i}}(j)+k}-z_{i}<xk\}\Big]\\
\nonumber
&= \prod_{i=1}^{\lfloor\frac{v_p n-k}{ck}\rfloor-1}\pr_{0}^{\omega}\Big[\bigcap_{j=1}^{\ell(z_{i},\tau_{z_{i+1}}-k)}\{X_{\tau_{z_{i}}(j)+k}-z_{i}<xk\}\Big]+\pr_{0}^{\omega}[\bar{B}_n(c,k)^{\complement}]\\
&\leq 
\prod_{i=1}^{\lfloor\frac{v_p n-k}{ck}\rfloor-1}\pr_{0}^{\omega}\Big[\bigcap_{j=1}^{\ell(z_{i},\tau_{z_{i+1}}-k)}\{X_{\tau_{z_{i}}(j)+k}-z_{i}<xk\}\Big]+\pr_{0}^{\omega}[B(v_p n,ck)],
\label{prodprob}
\end{align}
where we used that
$\bar{B}_n(c,k)\supset B(v_p n,ck)^\complement$.
We next control the main term in \eqref{prodprob}:
\begin{align*}
\pr_{0}^{\omega}\Big[\bigcap_{j=1}^{\ell(z_{i},\tau_{z_{i+1}}-k)}\{X_{\tau_{z_{i}}(j)+k}-z_{i}<xk\}\Big]&=\pr_{0}^{\theta^{z_i}\omega}\Big[\bigcap_{j=1}^{\ell(0,\tau_{z_1}-k)}\{X_{\tau_{0}(j)+k}<xk\}\Big].
\end{align*}
Using analogous computations as in \eqref{probDev} we obtain
\begin{align*}
\pr_{0}^{\omega}&\Big[\bigcap_{j=1}^{\ell(0,\tau_{z_1}-k)}\{X_{\tau_{j}(0)+k}<xk\}\Big]\\
&=1-\frac{\pw [\btau_0(1)\leq k,X_{\btau_0(1)}=xk]}{1-\pw[X_{\btau_0(1)}=0]-\pw[ \btau_0(1)> k,X_{\btau_0(1)}=xk]\pr^{\omega}_{xk} [\tau_0<\tau_{z_1}]}\\
&=1-\frac{\pr^{\hat{\omega}}_1[ \tau_{xk}\leq k]}{1-\pr^{\hat{\omega}}_1[ \tau_{xk}> k](1-S(xk,\omega)/S(ck,\omega))}\\
&=\frac{\pr^{\hat{\omega}}_1[ \tau_{xk}> k]S(xk,\omega)/S(ck,\omega)}{1-\pr^{\hat{\omega}}_1[ \tau_{xk}> k](1-S(xk,\omega)/S(ck,\omega))}.
\end{align*}
Now going back to the product in \eqref{prodprob} we have
\begin{align}
\nonumber\prod_{i=1}^{\lfloor\frac{v_p n-k}{ ck}\rfloor-1}&\pr_{0}^{\omega}\Big[\bigcap_{j=1}^{\ell(z_{i},\tau_{z_{i+1}}-k)}\{X_{\tau_{z_{i}}(j)+k}-z_{i}<xk\}\Big]\\
\nonumber &=\prod_{i=1}^{\lfloor\frac{v_p n-k}{ ck}\rfloor-1} \pr_{0}^{\theta^{z_i}\omega}\Big[\bigcap_{j=1}^{\ell(0,\tau_{z_{1}}-k)}\{X_{\tau_{0}(j)+k}<xk\}\Big]\\
\nonumber &=\prod_{i=1}^{\lfloor\frac{v_p n-k}{ ck}\rfloor-1}\frac{\pr^{\hat{\omega}}_1[ \tau_{xk}> k]S(xk,\omega)/S(ck,\omega)}{1-\pr^{\hat{\omega}}_1[ \tau_{xk}> k](1-S(xk,\omega)/S(ck,\omega))}\\
\label{prodprob3}&=\prod_{i=1}^{\lfloor\frac{v_p n-k}{ ck}\rfloor-1}\frac{f(\theta^{z_i}\omega,x,k) g(\theta^{z_i}\omega,x,k)}{1-f(\theta^{z_i}\omega,x,k)(1- g(\theta^{z_i}\omega,x,c,k))}.
\end{align}
The terms in the last product  of \eqref{prodprob3} are independent since
they depend on disjoint subsets of the environment. Moreover, by stationarity
they are identically distributed.
Integrating \eqref{prodprob3} with respect to $P$ we obtain
\begin{align*}
\int& \Big(\prod_{i=1}^{\lfloor\frac{v_p n-k}{ck}\rfloor-1}\frac{f(\theta^{z_i}\omega,x,k) g(\theta^{z_i}\omega,x,c,k)}{1-f(\theta^{z_i}\omega,x,k)(1- g(\theta^{z_i}\omega,x,c,k))}\Big) P(d\omega)\\
 &=\Big(\int\frac{f(\omega,x,k) g(\omega,x,c,k)}{1-f(\omega,x,k)(1- g(\omega,x,c,k))}P(d\omega)\Big)^{\lfloor\frac{v_p n-k}{ck}\rfloor-1}\\
 &=\Big(1-\int\frac{1-f(\omega,x,k)}{1-f(\omega,x,k)(1- g(\omega,x,c,k))}P(d\omega)\Big)^{\lfloor\frac{v_p n-k}{ ck}\rfloor-1}\\
&\leq \exp \Big(-{\big(\lfloor\frac{v_p n-k}{ ck}\rfloor-1\big)} \int\frac{1-f(\omega,x,k)}{1-f(\omega,x,k)(1- g(\omega,x,c,k))}P(d\omega) \Big).
\end{align*}
Taking expectations in \eqref{prodprob} and using the last estimate together
with \eqref{eq-noBT}, this yields \eqref{Pbound2}.
\end{proof}
\section{Environment partitioning, approximations, and proof of Lemma \ref{lemma2}}
\label{sec-lem13}
We introduce a partitioning of the interval 
$[0,ck-1]$ that will be useful when controlling
maxima of the potential using empirical fields. 
We then introduce approximations of various rate functions, and then 
provide the proof of Lemma \ref{lemma2}.

\subsection{Environment partitioning, basic LDP, and reverse environment}
We begin by introducing a partition of the environment into blocks.
\begin{df}[$\e$-partitioning]\label{partition} Choose $\e>0$ small  
  enough so that
  $x/\e$ is an integer. Divide the interval $[0,xk-1]\cap \mathbb{Z}$ into
  disjoint intervals $I_1,I_2,\ldots I_{x/\e}$ of approximate length $\e k$ 
  in the most even way possible, so for every $i,j$ we have $||I_i|-|I_j||\leq 1$. Define the intervals $\bar{I}_1,\bar{I}_2,\ldots, \bar{I}_{\lfloor(c-x)/\e\rfloor}$ in the same way as a partitioning for the interval $[xk,cx-1]\cap \mathbb{Z}$ (observe that since we are using the same value of $\e$ we cannot assure that $(c-x)/\e$ is an integer). For every interval we define its empirical field
\[
R_{i,\e}:=\frac{1}{|I_i|} \sum\limits_{j\in I_i} \delta_{\theta^j\omega} \quad \mbox{and}\quad \bar{R}_{i,\e}:=\frac{1}{|\bar{I}_i|} \sum\limits_{j\in \bar{I}_i} \delta_{\theta^j\omega}.
\]
\end{df}
Define for $m<n$
\[
R_n^m:=\frac{1}{n-m}\sum \limits_{j=m}^{n-1} \delta_{\theta^j\omega}.
\]
The next standard lemma exploits the product structure of $P=p^\mathbb{Z}$ to show a joint LDP for appropriate  vectors of empirical processes $(R_{n_i}^{m_i})_{i=1}^B$. We ommit the straightforward proof.
\begin{lemma}\label{blocksLDP} 
For any constants $0=s_0<s_1<s_2<\ldots<s_B=1$ in $[0,1]$, the vector of empirical fields 
 $(R_{s_1 n},R_{s_2 n}^{s_1 n},\ldots, R_{n}^{s_{B-1} n})$  satisfies, under $P=p^\mathbb{Z}$,
  a large deviation principle in $M_1(\Sigma)^{B}$, equipped with the product topology, with the rate function
\begin{align}
    I_{R}(\eta_1,\ldots,\eta_B):=\sum \limits_{i=1}^B (s_i-s_{i-1})h(\eta_i|P).
\end{align}
\end{lemma}

Now we make use of the blocks to estimate the value of $g$, as defined in \eqref{fgdef}, in terms of empirical fields. Define
 \[
 \delta_{\e}=\delta_{\e}(\{ R_{i,\e}\}_{i=1,\ldots,x/\e},\{ \bar{R}_{i,\e}\}_{i=1,\ldots,\lfloor(c-x)/\e\rfloor})
 \]
 as 
\begin{align*}
\delta_{\e}&=:\max \limits_{1\leq j\leq \lfloor(c-x)/\e\rfloor} \sum_{i=1}^j \int \log\rho_0(\omega) \bar{R}_{i,\e}(d\omega)-\max \limits_{1\leq j\leq x/\e} \sum_{i=1}^j \int \log\rho_0(\omega) R_{i,\e}(d\omega).
\end{align*}
 and
 \begin{equation}\label{Deltadef}
 \Delta_{\e}=\Delta_{\e}(\{ R_{i,\e}\}_{i=1,\ldots,x/\e},\{ \bar{R}_{i,\e}\}_{i=1,\ldots,\lfloor(c-x)/\e\rfloor})
:=\e\delta_{\e} -\e \sum \limits_{i=1}^{x/\e}\int \log \rho_0(\omega) R_{i,\e}(d\omega).
\end{equation}
Recall the constant $C_\kappa$, see \eqref{eq-ckappa}.
\begin{lemma}\label{lemma1} Suppose $\omega_x\in [\kappa,1-\kappa]$ for all $x$.
Then,  it holds that
\[
 -\frac{1}{k}\log\frac{S(xk,\omega)}{S(ck,\omega)}\leq \Big(\Delta_{\e}+\e\log C_{\kappa} \Big)^+ +\frac{(c+x)\log C_\kappa}{\e k}+\frac{\log (xk)}{k},
\]
and for large enough $k$,
\[
 -\frac{1}{k}\log\frac{S(xk,\omega)}{S(ck,\omega)}\geq \Big(\Delta_{\e}-\e\log C_{\kappa} \Big)^+ -\frac{(c+x)\log C_\kappa}{\e k}-\frac{\log (ck)}{k}.
\]
\end{lemma}
  
\begin{proof}
Using \eqref{eqS} we get
\begin{equation}
\label{eq-deltabound1}
 \log\frac{S(xk,\omega)}{S(ck,\omega)} \leq \max \limits_{0\leq j\leq xk-1} V_{\omega}(j)-\max \limits_{0\leq j\leq ck-1} V_{\omega}(j)+\log(xk),
\end{equation}
and
\begin{equation}
\label{eq-deltabound2}
\log\frac{S(xk,\omega)}{S(ck,\omega)}\geq \max \limits_{0\leq j\leq xk-1} V_{\omega}(j)-\max \limits_{0\leq j\leq ck-1} V_{\omega}(j)-\log(ck),
\end{equation}
Consider the partitioning in Definition~\ref{partition}. The size of each interval satisfies  $\epsilon k \leq |I_i|,|\bar{I}_i|\leq \epsilon k+1$ for every $i$ possible, hence 
\begin{align*}
    V_{\omega}(xk)&=\sum \limits_{j=0}^{xk-1} \log \rho_j(\omega)
    =\sum \limits_{i=1}^{x/\e} \sum \limits_{m\in I_i}\log \rho_m(\omega)\\
    &=\sum \limits_{i=1}^{x/\e} |I_i|\int \log \rho_0(\omega) R_{i,\e}\\
    &= \e k \sum \limits_{i=1}^{x/\e}\int \log \rho_0(\omega) R_{i,\e}+a_\e(x,k,\omega),
\end{align*}
where $a_\e(x,k,\omega)$ is the error by difference in interval lengths and have the following deterministic bounds
\[
-\frac{x}{\e}\log C_{\kappa}\leq a_\e(x,k,\omega)\leq \frac{x}{\e}\log C_{\kappa}.
\]
We can bound the maximum in terms of those empirical fields using the ellipticity of the environment. Observe that if $j\in I_m$ then it holds that 
\begin{align*}
 V_{\omega}(j)&=V_{\omega}(i_m)+\sum_{i=i_m}^{j}\log \rho_i(\omega)\\
 &\leq \sum \limits_{i=1}^m \sum \limits_{q\in I_i} \log \rho_q(\omega) +(\e k+1) \log C_{\kappa}\\
 &\leq (\e k+1) \sum \limits_{i=1}^m \int \log \rho_0(\omega) R_{i,\e}(d\omega) + (\e k+1)\log C_{\kappa}.
\end{align*}
Hence,
\begin{align*}
   & \frac{1}{k}\max \limits_{0\leq j\leq xk-1} V_{\omega}(j)=\frac{1}{k}\max \limits_{1\leq m\leq x/\e} \max\limits_{j\in I_m} V_{\omega}(j)\\
    &\leq \frac{(\e k+1)}{k}\max \limits_{1\leq m\leq x/\e} \sum \limits_{i=1}^m \int \log \rho_0(\omega) R_{i,\e}(d\omega)+\frac{(\e k+1)}{k}\log C_{\kappa},
\end{align*}
and 
\begin{align*}
  &  \frac{1}{k}\max \limits_{0\leq j\leq xk-1} V_{\omega}(j)\geq \frac{(\e k+1)}{k}\max \limits_{1\leq m\leq x/\e} \sum \limits_{i=1}^m \int \log \rho_0(\omega) R_{i,\e}(d\omega)-\frac{(\e k+1)}{k}\log C_{\kappa}.
\end{align*}
Analogous calculations hold for the empirical fields $\bar{R}_{i,\epsilon}$: 
\begin{align*}
&\frac{1}{k}\max \limits_{xk\leq j\leq ck-1} V_{\omega}(j)-V_{\omega}(xk)\\
&\leq \big(\e+\frac{1}{k}\big) \max \limits_{1\leq m\leq \lfloor(c-x)/\e\rfloor} \sum_{i=1}^m \int \log\rho_0(\omega) \bar{R}_{i,\epsilon}(d\omega)+ \big(\e+\frac{1}{k}\big) \log C_{\kappa},
\end{align*}
and
\begin{align*}
&\frac{1}{k}\max \limits_{xk\leq j\leq ck-1} V_{\omega}(j)-V_{\omega}(xk)\\
&\geq \big(\e+\frac{1}{k}\big) \max \limits_{1\leq m\leq \lfloor(c-x)/\e\rfloor} \sum_{i=1}^m \int \log\rho_0(\omega) \bar{R}_{i,\epsilon}(d\omega)-\big(\e+\frac{1}{k}\big) \log C_{\kappa}.
\end{align*}
By the definition it holds that $|\delta_\e|\leq c\log C_\kappa/\e$. Also for positive $b$  we have $(a+b)^+\leq a^++b$ and hence
\begin{align*}
    \Big(\Delta_{\e}+\frac{\delta_\e}{k}+\big(\e+\frac{x}{\e k}\big)\log C_{\kappa} \Big)^+\leq \Big(\Delta_{\e}+\e\log C_{\kappa} \Big)^++\frac{(c+x)\log C_\kappa}{\e k},
\end{align*}
and if $b<|a|$ it holds $(a-b)^+\geq a^+-b$. Therefore for large enough $k$
\begin{align*}
    \Big(\Delta_{\e}+\frac{\delta_\e}{k}-\big(\e+\frac{x}{\e k}\big)\log C_{\kappa} \Big)^+\geq \Big(\Delta_{\e}-\e\log C_{\kappa} \Big)^+-\frac{(c+x)\log C_\kappa}{\e k},
\end{align*}
Applying those bounds to \eqref{eq-deltabound1} and \eqref{eq-deltabound2} yields the lemma.
\end{proof}
Recall the notation $\hat \omega$ and $\hat \omega^L$ (see \eqref{eq-hatomega} and
\eqref{eq-TL}).
\begin{lemma}\label{lemmadrift} 

For any $x>L$ and $\eta\in M_1^s(\Sigma)$,
\[
\int \log \rho_0(\hat\omega^L) \eta(d\omega)\leq\int \log \rho_x(\hat{\omega}) \eta(d\omega) \leq -\left|\int \log \rho_0(\omega) \eta(d\omega)\right|,
\]
\end{lemma}

\begin{proof} 
  Since $\hat{\omega}^L_x\geq \hat{\omega}_x$ for all $x>L$, we have
  that $\rho_x(\hat{\omega})\geq \rho_x(\hat{\omega}^L)$,
  and the first inequality follows. To get the second one observe that by the definition of the transformed probabilities it holds that
\begin{align}
    \nonumber \int \log \rho_x(\hat{\omega}) \eta(d\omega) &= \int \log \frac{\rho_x(\omega)S(x-1,\omega)}{S(x+1,\omega)} \eta(d\omega)\\
    \label{logint} &=\int \log \rho_0(\omega)\eta(d\omega)+ \int \log \frac{S(x-1,\omega)}{S(x+1,\omega)} \eta(d\omega).
\end{align}
The second term satisfies
\begin{align*}
    \frac{S(x-1,\omega)}{S(x+1,\omega)}&=\frac{\sum\limits_{j=0}^{x-2}e^{V_\omega (j)} }{\sum\limits_{j=0}^{x}e^{V_\omega (j)}}
    =\frac{e^{V_\omega (x-2)}\sum\limits_{j=0}^{x-2}e^{V_\omega (j)-V_\omega (x-2)} }{e^{V_\omega (x)}\sum\limits_{j=0}^{x}e^{V_\omega (j)-V_\omega (x)}},
\end{align*}
but $V_\omega (j)-V_\omega (x)=-\sum_{i=j+1}^{x} \log \rho_i(d\omega)=V_{r(\theta^{x}\bar{\omega})}(x-j)$ and therefore, recalling the notation $r(\cdot)$ and $\bar \omega$ for the reversed and flipped environment, we have
\begin{align*}
    \frac{S(x-1,\omega)}{S(x+1,\omega)}&=\frac{S(x-1,r(\theta^{x-2}\bar{\omega})  }{\rho_{x}\rho_{x-1}S(x+1,r(\theta^{x}\bar{\omega}) )}.
\end{align*}
Integrating with respect to $\eta$ we obtain that
\begin{align}
    \nonumber &\int \log \frac{S(x-1,\omega)}{S(x+1,\omega)} \eta(d\omega)=\int \log \frac{S(x-1,r(\theta^{x-2}\bar{\omega})  }{\rho_{x}\rho_{x-1}S(x+1,r(\theta^{x}\bar{\omega}) )} \eta(d\omega)\\
    \nonumber&\quad =-2\int \log \rho_0(\omega) \eta(d\omega)+\int \log S(x-1,\bar{\omega}) \eta(d\omega)
    -\int \log S(x+1,\bar{\omega}) \eta(d\omega)\\
    \nonumber&\quad \leq -2\int \log \rho_0(\omega) \eta(d\omega),
\end{align}
where the second equality used
the stationarity of the measure,
while the inequality used that $S(x-1,\bar{\omega})\leq S(x+1,\bar{\omega})$, 
as the latter is a sum of positive terms. Substituting  the last display 
in \eqref{logint}, we obtain that
\[
\int \log \rho_x(\hat{\omega}) \eta(d\omega) \leq -\int \log \rho_0(\omega) \eta(d\omega).
\]
To conclude the proof one just needs to notice that $\hat{\omega}_x\geq \omega_x$ for all $x$. 
\end{proof}

\subsection{Approximate rate functions and LDP for conditioned environment}
We introduce the following functions
\begin{align*}
\phi_{i}(\hat{\omega},\lambda)&=\ex_i^{\hat{\omega}} [e^{\lambda\tau_{i+1}}],\\
\phi_{i,M}(\hat{\omega},\lambda)&=\ex_i^{\hat{\omega}} [e^{\lambda\tau_{i+1}}\ind[\tau_{i+1}<M]],\\
   \hat{\phi}_L^M(\lambda, \omega)&={\phi}^M(\lambda, \hat{\omega}^L)=\ex^{\hat{\omega}^L}_0[e^{\lambda \tau_1}\ind[\tau_1<M]],\\
   \hat{\phi}_L(\lambda, \omega)&={\phi}(\lambda, \hat{\omega}^L)=\ex^{\hat{\omega}^L}_0[e^{\lambda \tau_1}].
\end{align*}
Observe that by Lemma~\ref{lemmadrift} there are no concerns about $\tau_{i+1}$ being finite on those functions. Fix $1\leq J<xk$ integer and define
\begin{align*}
    \hat{I}_J(x,k,\omega)&=\sup \limits_{\lambda\leq 0} \Big\{\lambda- \frac{1}{k}\sum \limits_{i=J}^{xk-1} \log \phi_{i}(\hat{\omega},\lambda)\Big\},\\
    \hat{I}_{J,M}(x,k,\omega)&=\sup \limits_{\lambda\leq 0} \Big\{\lambda- \frac{1}{k}\sum \limits_{i=J}^{xk-1} \log \phi_{i,M}(\hat{\omega},\lambda)\Big\},\\
    \hat{I}_J^L(x,k,\omega)&=\sup \limits_{\lambda\leq 0} \Big\{\lambda- \frac{1}{k}\sum \limits_{i=J}^{xk-1} \log \hat{\phi}_L(\lambda, \theta^i\omega)\Big\},\\
    \hat{I}_{J,M}^L(x,k,\omega)&=\sup \limits_{\lambda\leq 0} \Big\{\lambda- \frac{1}{k}\sum \limits_{i=J}^{xk-1} \log \hat{\phi}_L^M(\lambda, \theta^i\omega)\Big\}.
\end{align*}
In general, we suppress the notations $k$ and $\omega$ from these functions when no confusion is possible, we also suppress $J$ in case it is $0$.

\begin{lemma}\label{lemmaPhat}  Suppose $ \omega_i\in [\kappa,1-\kappa]$ for all $i$.
Then, for any $J<xk$ (possibly depending on $k$), it holds that
\begin{align}
\label{eq-UBtau}
-\frac{1}{k}\log \pr^{\hat{\omega}}_J[\tau_{xk}\leq k]&\geq \hat{I}_J(x,k,\omega).
\end{align}
Further, there is  a constant $\alpha_3=\alpha_3(\kappa,x)>0$ so that
for any $M>0$ and $l\in (0,1)$,
\begin{align}
\label{eq-LBtau}
-\frac{1}{k}\log \pr^{\hat{\omega}}_J[\tau_{xk}\leq k(1+l)]&\leq \hat{I}_{J,M}(x,k,\omega)+\alpha_3l-\frac{1}{k}\log \big(1-2e^{-\frac{2kl^2}{xM^2}}\big).
\end{align}
\end{lemma}
\begin{proof}
Let $\tau_+(i)=\inf \{ t>0: X_{\tau_{i}+t}=i+1\}$ denote the time to hit $i+1$ from $i$.
We can decompose $\tau_{xk}$ as the sum of such variables and therefore for $\lambda\leq 0$ we have  from Chebyshev's inequality and the strong Markov property
that
\begin{align*}
    \pr^{\hat{\omega}}_J[\tau_{xk}\leq k]&=\pr^{\hat{\omega}}_J\Big[\sum \limits_{i=J}^{xk-1}\tau_+(i) \leq k \Big]\\
    &\leq \exp\Big( \inf \limits_{\lambda \leq 0 } \Big\{\sum\limits_{i=J}^{xk-1} \log\ex_i^{\hat{\omega}} [e^{\lambda\tau_{i+1}}\ind[\tau_{i+1}<\infty]] -\lambda k\Big\}\Big)\\
    &\leq \exp\Big(-k \sup \limits_{\lambda \leq 0 } \Big\{\lambda -\frac{1}{k}\sum\limits_{i=J}^{xk-1} \log\ex_i^{\hat{\omega}} [e^{\lambda\tau_{i+1}}]\Big\}\Big),
\end{align*}
which implies \eqref{eq-UBtau}.

We turn to the  proof of \eqref{eq-LBtau}. Let
\begin{equation}
  \label{eq-lms}
F_M(\lambda,x, k,\hat\omega)=\lambda -\frac{1}{k}\sum\limits_{i=J}^{xk-1} \log\phi_{i,M}(\hat{\omega},\lambda), \;\lambda_M^*(x)=\argmax \limits_{\lambda \leq 0 } F_M(\lambda,x, k,\hat
\omega).
\end{equation}
Since 
\begin{align*}
F_M(0,x,k,\hat\omega)&=-\frac1k \sum_{i=J}^{xk-1} \log \pr_i^{\hat \omega}[\tau_{i+1}<M]\\
&<
-\frac1k \sum_{i=J}^{xk-1} \log \pr_i^{\hat \omega}[\tau_{i+1}=1]=\lim_{\lambda\to-\infty}F_M(\lambda,x,k,\hat\omega),
\end{align*}
and $E_i^{\hat \omega}(\tau_{i+1}|\tau_{i+1}<M)\geq 1$ with strict inequality 
when $i>1$, it follows that $\lambda_M^*(x)\neq 0$ is well defined.
Now  let $A_{\omega,M}=\{\tau_+(i)\leq M,\, \mbox{ for all }i\in [J,xk-1]\}$ and denote  by
$\tilde{\pr}^{\omega,M}$ the quenched law of $\{\tau_+(i) \}_{J\leq i<xk}$ conditioned on the event $A_{\omega,M}$. Fix $l>0$ and  write 
\begin{align}
\nonumber    \pr^{{\hat{\omega}}}_J[ \tau_{xk}< k(1+l)]&\geq \pr^{\hat{\omega}}_J[ \tau_{xk}< k(1+l)\mid A_{\omega,M}] \pr^{\hat{\omega}}_J[A_{\omega,M}]\\
\nonumber    &\geq \tilde{\pr}^{{{\hat{\omega}},M}}_J\big[k(1-l)< \tau_{xk}< k(1+l)\big]\pr^{\hat{\omega}}_J[A_{\omega,M}]\\
\nonumber    &=\tilde{\pr}^{{{\hat{\omega}},M}}_J\big[k(1-l)< \sum_{i=J}^{xk-1}\tau_+(i) < k(1+l)\big]\pr^{\hat{\omega}}_1[A_{\omega,M}]\\
\label{Phat2}    &=\int \limits \ind[(z_i)_i:k(1-l) <\sum_i z_i<k(1+l)]\tilde{\pr}^{\hat{\omega},M}(d\mathbf{z})\pr^{\hat{\omega}}_J[A_{\omega,M}].
\end{align}
 Since $\tilde{\pr}^{{\hat{\omega}},M}$ is a product measure we can write
\begin{equation}
\label{prodmeasure1}\tilde{\pr}^{\hat{\omega},M}(d\mathbf{z})=\prod \limits_{i=J}^{xk-1}\overline{\pr}^{i,{\hat{\omega}},M}(dz_i),
\end{equation}
where $\overline{\pr}^{i,{\hat{\omega}},M}$ corresponds to the quenched law of $\tau_+(i)$ under $A_{\omega,M}$. Define the tilted measure
$\overline{\qr}^{i,{\hat{\omega}},M}$ by 
\[
\overline{\qr}^{i,{\hat{\omega}},M}(dz)=\frac{\pr^{\hat{\omega}}_J[\tau_+(i)<M]e^{\lambda_M^*(x)z}}{\phi_{i,M}(\hat{\omega},\lambda_M^*(x))}\overline{\pr}^{i,{\hat{\omega}},M}(dz).
\]
Since 
\[
\int e^{\lambda_M^*(x)z}\overline{\pr}^{i,\hat\omega,M}(dz)=\frac{\ex_J^{\hat{\omega}}[e^{\lambda^*(x)\tau_+(i)}\ind[\tau_+(i)<M]]}{\pr^{\hat{\omega}}_J[\tau_+(i)<M]}=\frac{\phi_{i,M}({\hat{\omega}}, \lambda_M^*(x))}{\pr^{\hat{\omega}}_J[\tau_+(i)<M]},
\]
then $\overline{\qr}^{i,\hat\omega,M}$ is indeed a probability measure. We also consider the joint product measure $\qr^{\hat \omega,M}$ in a way analogous to \eqref{prodmeasure1}:
\begin{align*}
\qr^{\hat \omega,M}(d\mathbf{z})&=\prod \limits_{i=J}^{xk-1}\overline{\qr}^{i,\hat \omega,M}(dz_i)\\
&=\frac{\pr^{\hat{\omega}}_J[A_{\omega,M}]e^{\lambda^*_M(x)\sum_i z_i}}{\prod \limits_{i=J}^{xk-1}\phi_{i,M}(\hat{\omega}, \lambda_M^*(x))}\tilde{\pr}^{\hat \omega,M}(d\mathbf{z}),
\end{align*}
see \eqref{eq-lms} for the definition of $\lambda_M^*(x)$.
Introduce the set
$C(z)=\{k(1-l) <z<k(1+l) \}$. From \eqref{Phat2} we get
\begin{align}
    \nonumber&\pr^{{\hat{\omega}}}_J[ \tau_{xk}< k(1+l)]\geq \int \limits \ind[C(\sum_iz_i)]\tilde{\pr}^{\hat\omega,M}(d\mathbf{z})\pr^{\hat{\omega}}_J[A_{\omega,M}]\\
   \nonumber &=\frac{\prod \limits_{i=J}^{xk-1}\phi_{i,M}(\hat{\omega}, \lambda_M^*(x))}{e^{\lambda_M^*(x)k(1-l)}}\int \frac{\pr^{{\hat \omega}}_1[A_{\omega,M}]e^{\lambda_M^*(x)k(1-l)}}{\prod \limits_{i=J}^{xk-1}\phi_{i,M}(\hat{\omega}, \lambda_M^*(x))} \ind[C(\sum_iz_i)]\tilde{\pr}^{\hat \omega,M}(d\mathbf{z})\\
    \nonumber&\geq \frac{\prod \limits_{i=J}^{xk-1}\phi_{i,M}(\hat{\omega}, \lambda_M^*(x))}{e^{\lambda_M^*(x)k(1-l)}}\int  \frac{\pr^{\hat {\omega}}_1[A_{\omega,M}]e^{\lambda_M^*(x)\sum_i z_i}}{\prod \limits_{i=J}^{xk-1}\phi_{i,M}(\hat{\omega}, \lambda_M^*(x))}\ind[C(\sum_iz_i)]\tilde{\pr}^{\hat\omega,M}(d\mathbf{z})\\
    \nonumber&\geq \frac{\prod \limits_{i=J}^{xk-1}\phi_{i,M}(\hat{\omega}, \lambda_M^*(x))}{e^{\lambda_M^*(x)k(1-l)}}\int \ind[C(\sum_iz_i)]\qr^{\hat \omega,M}(d\mathbf{z})\\
    \label{upperIM}&= e^{-k\hat{I}_{J,M}(x,k,\omega)}e^{l\lambda_M^*(x)k}\qr_J^{\hat \omega,M}\Big[k(1-l)< \sum_{i=J}^{xk-1}\tau_+(i) < k(1+l)\Big].
\end{align}
Now we prove that
\begin{equation}
\label{Qbound}\qr_J^{\hat \omega,M}\Big[k(1-l)< \sum_{i=J}^{xk-1}\tau_+(i) < k(1+l)\Big]\geq 1-2e^{-\frac{2kl^2}{xM^2}}.
\end{equation}
 First observe that the moment generating function of $\tau_+(i)$ under $\overline{\qr}^{i,\hat\omega,M}$ is 
\begin{align*}
\phi^{\overline{\qr}}_{i,M}(\hat{\omega}, \lambda)&=\int e^{\lambda z} \overline{\qr}^{i,\hat{\omega},M}(dz)\\
&=\pr^{\hat{\omega}}_J[\tau_+(i)<M]\int \frac{e^{(\lambda+\lambda_M^*(x)) z}}{\phi_{i,M}(\hat{\omega},\lambda^*_M(x))} \overline{\pr}^{i,\hat{\omega},M}(dz)\\
&=\frac{\phi_{i,M}(\hat{\omega},\lambda_M^*(x)+\lambda)}{\phi_{i,M}(\hat{\omega},\lambda_M^*(x))}.
\end{align*}
Therefore
\[
\int \tau_+(i)  d\overline{\qr}^{i,\hat{\omega},M}=\frac{\phi_{i,M}'(\hat{\omega},\lambda_M^*(x))}{\phi_{i,M}(\hat{\omega},\lambda_M^*(x))},
\]
and thus 
\[
\int \sum \limits_{i=J}^{xk-1}\tau_+(i)  d{\qr^{\hat \omega,M}}=\sum \limits_{i=J}^{xk-1}\frac{\phi_{i,M}'(\hat{\omega},\lambda_M^*(x))}{\phi_{i,M}(\hat{\omega},\lambda_M^*(x))}.
\]
Since $\lambda_M^*(x)\in (-\infty,0)$,  it is a critical point of the function $F_M$, i.e.
\begin{align*}
    &\frac{\partial}{\partial \lambda}\Big(\lambda- \frac{1}{k}\sum \limits_{i=J}^{xk-1} \log \phi_{i,M}(\hat{\omega},\lambda) \Big)\Big|_{\lambda=\lambda^*(x)}=0
\Leftrightarrow\sum  \limits_{i=J}^{xk-1}\frac{\phi_{i,M}'(\hat{\omega},\lambda_M^*(x))}{\phi_{i,M}(\hat{\omega},\lambda_M^*(x))}=k,
\end{align*}
implying that
\[
\int \sum \limits_{i=J}^{xk-1}\tau_+(i)  d{\qr^{\hat\omega,M}}=k.
\]
Since under $\qr^{\hat \omega,M}$ we have $\tau_+(i)<M$ for all $i$, and $\qr^{\hat \omega,M}$ is (quenched) a product measure, we  can use 
Hoeffding's inequality (see e.g. \cite[Corollary 2.4.7]{DZ})
to conclude \eqref{Qbound}.

We next find an lower bound for $\lambda^*_M(x)$. Toward this end, recall that  $\lambda^*_{M}(x)$ is the solution of the equation
\[
\sum  \limits_{i=J}^{xk-1}\frac{\phi_{i,M}'(\hat{\omega},\lambda)}{\phi_{i,M}(\hat{\omega},\lambda)}=k,
\]
 but since $\lambda\leq 0$,
 \begin{align*}
     \frac{\phi_{i,M}'(\hat{\omega},\lambda)}{\phi_{i,M}(\hat{\omega},\lambda)}&=\frac{\ex^{\hat{\omega}}_i[\tau_{i+1}e^{\lambda  \tau_{i+1}}\ind[\tau_{i+1}<M]]}{\ex^{\hat{\omega}}_i[e^{\lambda  \tau_{i+1}}\ind[\tau_{i+1}<M]}\\
     &\leq \frac{\hat{\omega}_i e^\lambda +\ex^{\hat{\omega}}_i[\tau_{i+1}e^{\lambda  \tau_{i+1}}\ind[2\leq \tau_{i+1}<M]]}{\hat{\omega}_i e^\lambda }\\
      &\leq 1+\frac{e^{\lambda}}{\hat{\omega}_i }\leq 1+\frac{e^{\lambda}}{\kappa},
 \end{align*}
thus
 \begin{align}
    \nonumber k&=\sum  \limits_{i=J}^{xk-1}\frac{\phi_{i,M}'(\hat{\omega},\lambda^*_{M}(x))}{\phi_{i,M}(\hat{\omega},\lambda^*_{M}(x))}\leq (xk-J)\Big(1+\frac{e^{\lambda^*_{M}(x)}}{\kappa}\Big)\\
    \nonumber&\leq xk\Big(1+\frac{e^{\lambda^*_{M}(x)}}{\kappa}\Big)\\
    \label{lambdabound}&\Rightarrow e^{\lambda^*_{M}(x)}\geq \kappa\Big(\frac{1-x}{x}\Big)\Rightarrow \lambda^*_{M}(x) \geq \log  \Big(\kappa\Big(\frac{1-x}{x}\Big)\Big).
 \end{align}
 Combining \eqref{lambdabound}, \eqref{Qbound} with \eqref{upperIM} yields \eqref{eq-LBtau}.
\end{proof}

\begin{lemma}
\label{lemmaShat} For any $a,b\in \mathbb{Z}^+$ it holds that
 \begin{align*}
    S(a,\theta^b\hat{\omega})&= \frac{S(b+1,\omega)S(b,\omega)}{e^{V_{\omega}(b)}}\Big(\frac{1}{S(b,\omega)}-\frac{1}{S(a+b,\omega)}\Big).
 \end{align*}
\end{lemma}

\begin{proof}
 By the definition of $S$ we have 
 \[
S(a,\theta^b\hat{\omega})=\sum \limits_{i=0}^{a-1} e^{V_{\theta^b\hat{\omega}}(i)}.
 \]
 Since $\rho_x(\hat{\omega})=\rho_x(\omega)S(x-1,\omega)/S(x+1,\omega) $,
 we get
 \begin{align}
     \nonumber V_{\theta^b\hat{\omega}}(i)&=\sum \limits_{j=1}^{i} \log \rho_j(\theta^b\hat{\omega})
    =\sum \limits_{j=1}^{i} \log \rho_{j+b}(\hat{\omega})\\
     \nonumber&=\sum \limits_{j=1}^{i} \Big(\log \rho_{j+b}({\omega})+\log \frac{S(j+b-1,\omega)}{S(j+b+1,\omega)}\Big)\\
     \nonumber&=\sum \limits_{j=b+1}^{i+b} \log \rho_{j}({\omega})+\log \frac{S(b-1,\omega)S(b,\omega)}{S(i+b-1,\omega)S(i+b,\omega)}\\
     \label{Vhat1}&=V_{\omega}(i+b)-V_{\omega}(b)+\log \frac{S(b+1,\omega)S(b,\omega)}{S(i+b+1,\omega)S(i+b,\omega)}.
 \end{align}
 Hence
 \begin{align*}
     S(a,\theta^b\hat{\omega})=\frac{S(b+1,\omega)S(b,\omega)}{e^{V_{\omega}(b)}}\sum \limits_{i=0}^{a-1} \frac{e^{V_{\omega}(i+b)}}{S(i+b+1,\omega)S(i+b,\omega)},
 \end{align*}
but observe that 
\begin{align*}
\frac{e^{V_{\omega}(i+b)}}{S(i+b+1,\omega)S(i+b,\omega)}&=\frac{1}{S(i+b,\omega)}-\frac{1}{S(i+b+1,\omega)}.
\end{align*}
Thus we have a telescopic sum and
 \begin{align*}
     S(a,\theta^b\hat{\omega})&= \frac{S(b+1,\omega)S(b,\omega)}{e^{V_{\omega}(b)}}\Big(\frac{1}{S(b,\omega)}-\frac{1}{S(a+b,\omega)}\Big),
 \end{align*}
 concluding the proof.
\end{proof}

\begin{lemma}\label{lemmaptau} Suppose 
  $ \omega_i\in [\kappa,1-\kappa]$ for all $i$.
  Then, for any $a>0$ so that $a\log M$ is an integer,
  \begin{equation}
    \label{eq-mrs3}
\frac{1}{k}\sum \limits_{i=1}^{xk-1} \pr_i^{\hat{\omega}}[\tau_{i+1}\geq M] \leq  xe^{-M^{1+a\log \kappa}/(a\log M)}+\frac{C_\kappa}{k}\sum \limits_{i=1}^{xk-1} \xi_i(a\log M,\omega).
\end{equation}
Moreover, if $\eta\in M_1^e(\Sigma)$ then
\begin{equation}
  \label{eq-mrs4}
  \lim \limits_{M\to \infty}\lim \limits_{k\to \infty}\frac{1}{k}\sum \limits_{i=1}^{xk-1} \pr_i^{\hat{\omega}}[\tau_{i+1}\geq M]=0, \qquad \eta-\mbox{a.s.}.
\end{equation}
\end{lemma}

\begin{proof} Fix $a>0$ and observe that
\begin{align}
    \nonumber&\pr_i^{\hat{\omega}}[\tau_{i+1}\geq M]\\
    \nonumber &=\pr_i^{\hat{\omega}}[\tau_{i+1}\geq M, \tau_{i+1}< \tau_{i-a\log M}]+\pr_i^{\hat{\omega}}[\tau_{i+1}\geq M, \tau_{i+1}> \tau_{i-a\log M}]\\
    \label{ptau1}&\leq \pr_i^{\hat{\omega}}[\tau_{\{i+1,i-a\log M\}}\geq M]+\pr_i^{\hat{\omega}}[\tau_{i+1}> \tau_{i-a\log M}].
\end{align}
We deal with each term in the right hand side of 
\eqref{ptau1} separately. Concerning the first one, note that
from any point $m$ inside the interval $[i-a\log M, i+1]$ we can exit in $a\log M$ steps to the right. 
Recalling that  $\hat \omega_i\geq \omega_i$, see
\eqref{eq-hatomega} and using the 
ellipticity bound, it holds that
\[
\pr_m^{\hat{\omega}}[\tau_{\{i+1,i-a\log M\}}\leq a\log M ] \geq \kappa^{a\log M},
\]
and therefore, by the Markov property,
\begin{equation}
  \label{eq-mrs5}
\pr_m^{\hat{\omega}}[\tau_{\{i+1,i-a\log M\}}\geq M ]\leq \Big(1- \kappa^{a\log M}\Big)^{M/(a\log M)}\leq e^{-M^{1+a\log \kappa}/(a\log M)}.
\end{equation}

Turning to the second term in the right hand side of \eqref{ptau1}, we have
\begin{align*}
    \pr_i^{\hat{\omega}}[\tau_{i+1}> \tau_{i-a\log M}]&=\pr_{a\log M}^{\theta^{i-a\log M}\hat{\omega}}[\tau_{a\log M+1}> \tau_{0}]\\
    &=1-\frac{S(a\log M,\theta^{i-a\log M}\hat{\omega})}{S(1+a\log M,\theta^{i-a\log M}\hat{\omega})}\\
    &=\frac{e^{V_{\theta^{i-a\log M}\hat{\omega}}(a\log M)}}{S(1+a\log M,\theta^{i-a\log M}\hat{\omega})}.
\end{align*}
By Lemma~\ref{lemmaShat} we have 
\begin{align*}
    &S(1+a\log M,\theta^{i-a\log M}\hat{\omega})=\frac{S(i-a\log M+1,\omega)S(i-a\log M,\omega)}{e^{V_{\omega}(i-a\log M)}}\\
    &\qquad \qquad \qquad \qquad \qquad\qquad \qquad\qquad \times \Big(\frac{1}{S(i-a\log M,\omega)}-\frac{1}{S(i+1,\omega)}\Big)\\
    &=\frac{S(i-a\log M+1,\omega)(S(i+1,\omega)-S(i-a\log M,\omega))}{e^{V_{\omega}(i-a\log M)}S(i+1,\omega)}.
\end{align*}
Also by \eqref{Vhat1} we have
\begin{align*}
&e^{V_{\theta^{i-a\log M}\hat{\omega}}(a\log M)}\\
&=e^{V_{\omega}(i)-V_{\omega}(i-a\log M)}\frac{S(i-a\log M+1,\omega)S(i-a\log M,\omega)}{S(i+1,\omega)S(i,\omega)}.
\end{align*}
Hence,
\begin{align*}
    &\pr_i^{\hat{\omega}}[\tau_{i+1}> \tau_{i-a\log M}]\\
		&=e^{V_{\omega}(i)-V_{\omega}(i-a\log M)}\frac{S(i-a\log M+1,\omega)S(i-a\log M,\omega)}{S(i+1,\omega)S(i,\omega)}\\
    &\quad \times \frac{e^{V_{\omega}(i-a\log M)}S(i+1,\omega)}{S(i-a\log M+1,\omega)(S(i+1,\omega)-S(i-a\log M,\omega))}\\
    &=e^{V_{\omega}(i)}\frac{S(i-a\log M,\omega)}{S(i,\omega)(S(i+1,\omega)-S(i-a\log M,\omega))}.
\end{align*}
Using \eqref{Sdecomp}, we have that 
$$S(i+1,\omega)=S(i-a\log M,\omega)+e^{V_{\omega}(i-a\log M)}S(a\log M+1,\theta^{i-a\log M}\omega),$$ we get
\begin{align*}
\nonumber    e^{V_{\omega}(i)}&\frac{S(i-a\log M,\omega)}{S(i,\omega)(S(i+1,\omega)-S(i-a\log M,\omega))}\\
\nonumber    &=\frac{S(i-a\log M,\omega)}{S(i-a\log M,\omega)+e^{V_{\omega}(i-a\log M)}S(a\log M,\theta^{i-a\log M}\omega)}\\
&\quad \times \frac{e^{V_{\theta^{i-a\log M}\omega}(a\log M)}}{S(a\log M+1,\theta^{i-a\log M}\omega)}\\
\nonumber &\leq \Big(\frac{1-\kappa}{\kappa}\Big)\frac{{W}(a\log M+1,\theta^{i-a\log M}\omega)}{1+W(i-a\log M ,\omega){S}(a\log M,\theta^{i-a\log M}\omega)}\\
\nonumber &=C_{\kappa}\xi_i(a\log M,\omega),
\end{align*}
where we recall the constant $C_\kappa$ and the random variable $\xi_i(a\log M,\omega)$, see \eqref{eq-ckappa} and \eqref{defxi}.
Substituting this and \eqref{eq-mrs5}
in \eqref{ptau1} yields \eqref{eq-mrs3}.

Turning to the proof of \eqref{eq-mrs4}, suppose that $\mu=\int \log \rho d\eta > 0$. According to Lemma~\ref{Wbound},
\[
\xi_i(a\log M,\omega)\leq \frac{C_\kappa}{1+W(i-a\log M ,\omega){S}(a\log M,\theta^{i-a\log M}\omega)}.
\]
Abbreviate
$$W_i(a\log M)=W(i-a\log M ,\omega), \quad
S_i(a\log M)={S}(a\log M,\theta^{i-a\log M}\omega).$$
For any $J$ we have 
\begin{align}
\label{doublesum}
    \frac{1}{k}&\sum \limits_{i=J}^{xk-1}\xi_i(a\log M,\omega)\\
    \nonumber&=\frac{1}{k}\sum \limits_{i=J}^{xk-1}\xi_i(a\log M,\omega)\ind[W_i(a\log M)\geq M^{-\mu/2}]\\
\nonumber     &\quad \quad+\frac{1}{k}\sum \limits_{i=J}^{xk-1}\xi_i(a\log M,\omega)\ind [W_i(a\log M)<M^{-\mu/2}]\\
    &\leq \frac{C_\kappa}{k}\sum \limits_{i=J}^{xk-1}\Big(1+M^{-\mu/2}S_i(a\log M) \Big)^{-1}+\frac{C_\kappa}{k}\sum \limits_{i=J}^{xk-1}\ind [W_i(a\log M)<M^{-\mu/2}].\nonumber
\end{align}
While the first term in the right hand side of \eqref{doublesum} is a
Ces\'aro average to which the ergodic theorem can be applied,
the second term
needs more work. Observe that 
\begin{align}
   \nonumber\frac{1}{W_i(a\log M)}&=\frac{\sum \limits_{j=0}^{i-a\log M-1}
   e^{V_\omega(j)}}{e^{V_\omega (i-a\log M)}}
   =\sum \limits_{j=0}^{i-a\log M-1}e^{V_\omega(j)-V_\omega (i-a\log M)}\\
   &\leq \sum \limits_{j=-\infty}^{-1}e^{V_{\theta^{i-a\log M}\omega}(j)}
    \label{Winverse}=S(-\infty,\theta^{i-a\log M}\omega).
\end{align}
Thus,
\begin{align}
    \label{ergosum}\frac{1}{k}\sum \limits_{i=J}^{xk-1}\ind [W_i(a\log M)<M^{-\mu/2}]\leq \frac{1}{k}\sum \limits_{i=J}^{xk-1}\ind [S(-\infty,\theta^{i-a\log M}\omega)>M^{\mu/2}],
\end{align}
and now we can apply the ergodic theorem to 
the right hand side of \eqref{ergosum}. 
Taking the limit in $k$ we obtain 
\begin{align}
    \nonumber&\lim \limits_{k\to \infty}\frac{1}{xC_\kappa^2k}\sum \limits_{i=1}^{xk-1}\pr_i^{\hat{\omega}}[\tau_{i+1}> \tau_{i-a\log M}]\\
    \nonumber&\leq \lim \limits_{k\to \infty}\frac{1}{k}\sum \limits_{i=J}^{xk-1}\Big( \Big(1+M^{-\mu/2}S_i(a\log M) \Big)^{-1}+\ind [S(-\infty,\theta^{i-a\log M}\omega)>M^{\mu/2}]\Big)\\
    \label{twoterms}&=\int \Big(1+\e'M^{-\mu/2}S(a\log M,\omega) \Big)^{-1} \eta(d\omega)+  \eta(S(-\infty,\cdot)>M^{\mu/2}).
\end{align}
To take the limit in $M$ we 
deal with the terms separately. For the intergral term 
in the right hand side of \eqref{twoterms},
define $N_1(\omega)=\inf \{n: V_\omega(i)/i \geq (3/4)\mu \mbox{ for }i\geq n \}$ 
and observe that
according to the ergodic theorem,  $\eta (N_1(\cdot)>K)\to 0$ as $K\to\infty$.
Now,
\begin{align}
    \nonumber\int &\Big(1+M^{-\mu/2}S(a\log M,\omega) \Big)^{-1} \eta(d\omega)\leq \int \Big(1+M^{-\mu/2}e^{V_\omega(a\log M-1)} \Big)^{-1}\eta(d\omega)\\
    \nonumber&=\int \frac{\ind[N_1(\omega)\leq a\log M]}{1+M^{-\mu/2}e^{V_\omega(a\log M-1)}}\eta(d\omega)+\int \frac{\ind[N_1(\omega)> a\log M]}{1+M^{-\mu/2}e^{V_\omega(a\log M-1)}}\eta(d\omega)\\
    \label{Wbound2}&\leq\frac{1}{1+M^{\mu/4}}+\eta(N_1(\cdot)> a\log M-1),
\end{align}  
which tends to zero as $M\to \infty$.

For the second term in the right hand side of  \eqref{twoterms},
define $N_2(\omega)=\inf \{n: V_\omega(-i)/i \leq -(3/4)\mu \mbox{ for }i\geq n \}$ and observe that for $\alpha_4=(1-e^{-(3/4)\mu})^{-1}$ it holds that
\begin{align}
\nonumber S(-\infty,\omega)&=\sum \limits_{j=1}^{\infty}e^{V_\omega(-j)}
=\sum \limits_{j=1}^{N_2(\omega)-1}e^{V_\omega(-j)}+\sum \limits_{j=N_2(\omega)}^{\infty}e^{V_\omega(-j)}\\
\nonumber&\leq \frac{C_\kappa^{N_2(\omega)}-1}{C_\kappa-1}+\frac{e^{-N_2(\omega)(3/4)\mu}}{1-e^{-(3/4)\mu}}
\leq \frac{C_\kappa^{N_2(\omega)}}{C_\kappa-1}+\alpha_4,
\end{align}
and for some constant $\alpha_5$ depending only on $\kappa$ we have
\begin{align}
    \nonumber\eta(S(-\infty,\omega)>M^{\mu/2})&\leq \eta\Big(\frac{C_\kappa^{N_2(\omega)}}{C_\kappa-1}+\alpha_4>M^{\mu/2} \Big)\\
    \label{boundSminus}&= \eta\Big(N_2(\cdot)>\alpha_5+\frac{\log(M^{\mu/2}-\alpha_4)}{\log C_\kappa} \Big),
\end{align}
which tends to zero as $M\to \infty$ due to the ergodic theorem. We
conclude that if $\mu=\mu(\eta) > 0$ then
\[
\lim \limits_{M\to \infty} \lim \limits_{k\to \infty} \frac{1}{k}\sum \limits_{i=1}^{xk-1} \pr_i^{\hat{\omega}}[\tau_{i+1}> \tau_{i-a\log M}] =0, \quad 
\eta-\mbox{a.s.},
\]
completing the proof of \eqref{eq-mrs4} in that case.

Now consider $\mu=\int \log \rho_0 d\eta \leq 0$ and observe that 
\begin{equation}\label{eq-mrs6}
  \xi_i(a\log M,\omega)\leq {W}(a\log M+1,\theta^{i-a\log M}\omega).
\end{equation}
By the ergodic theorem and \eqref{eq-mrs6} we get that, $\eta$-a.s.,
\begin{align}
    \nonumber\lim \limits_{k\to \infty }\frac{1}{k}\sum \limits_{i=J}^{xk-1}\xi_i(a\log M,\omega)&\leq   \lim \limits_{k\to \infty }\frac{1}{k}\sum \limits_{i=J}^{xk-1}{W}(a\log M+1,\theta^{i-a\log M}\omega)\\
    \label{ergoW}&=x \int {W}(a\log M+1,\omega) \eta(d\omega).
\end{align}
Now observe that for any  $K\geq 1$,
\begin{align*}
    {W}(K+1,\omega)&=\frac{e^{V_\omega(K+1)}}{\sum \limits_{j=0}^{K}e^{V_\omega (j)}}
    =\!\Big(\sum \limits_{j=0}^{K}e^{V_\omega (j)-V_\omega(K+1)} \Big)^{-1}
\!  \!\! = \Big(S(-K-1,\theta^{K+1}\omega) \Big)^{-1}\!\!\!.
\end{align*}
By  the stationarity of the environment, for each $K$,
$S(-K-1,\omega)$ has the same distribution as $S(K+2,r(\overline{\omega}))-1$
where $r(\overline{\omega})$ was defined in \eqref{eq-mrs2}, and therefore
we obtain that
\begin{align}
    \nonumber\int {W}(a\log M+1,\omega) \eta(d\omega)&=\int \Big(S(-a\log M-1,\theta^{a\log M+1}\omega) \Big)^{-1} \eta(d\omega)\\
    \label{Wintegral}&=\int \frac{1}{S(a\log M+2,r(\overline{\omega}))-1} \eta(d\omega).
\end{align}
From the
proof of Theorem~2.1.1. from \cite{coursezeitouni} (originally \cite[Theorem (1.7)]{solomon}) we have that $\mu \leq 0$ implies that
$S(n,\omega)\to \infty$ $\eta$-a.s. as $n\to \infty$. Observe that $\overline{\omega}$ is an environment with 
\[
\int \log \rho_0(\overline{\omega}) \eta(d\omega)=
-\int \log \rho_0 \eta(d\omega) \geq 0,
\]
hence $S(a\log M+2,r(\overline{\omega}))\to_{M\to\infty} \infty$,
$\eta$-a.s.
Since $S(n,\omega)\geq 1+C_\kappa^{-1}$, then the function integrated in \eqref{Wintegral} is bounded by $C_\kappa$ and therefore by the
dominated convergence theorem we get 
\begin{align*}
    \lim \limits_{M\to \infty} \int {W}(a\log M+1,\omega) \eta(d\omega)=\lim \limits_{M\to \infty}\int \frac{1}{S(a\log M+2,r(\overline{\omega}))-1} \eta(d\omega) = 0,
\end{align*}
so for $\mu \leq 0$ it holds that 
\[
\lim \limits_{M\to \infty} \lim \limits_{k\to \infty} \frac{1}{k}\sum \limits_{i=1}^{xk-1} \pr_i^{\hat{\omega}}[\tau_{i+1}> \tau_{i-a\log M}] =0.
\]
Having proved \eqref{eq-mrs4} also in the case $\mu=\mu(\eta)\leq 0$,
the proof of the lemma is complete.
\end{proof}

\begin{lemma}\label{lemmaapproxPhat} Let $\omega_i\in[\kappa,1-\kappa]$ for all
  $i$.
 Then there exist  $\alpha_6,\alpha_7>0$ depending only on $\kappa$, $x$ so that
 for any $M+L<J$,
\begin{align}
  \label{deltaPM}    |\hat{I}_J(x,k,\omega)-\hat{I}_{J,M}(x,k,\omega)|\leq \frac{\alpha_6}{k}\sum \limits_{i=J}^{xk-1}\pr^{\hat{\omega}}_i[\tau_{i+1}\geq M],
\end{align}
\begin{align}
\label{deltaLM}    |\hat{I}_{J,M}^L(x,k,\omega)-\hat{I}_J^L(x,k,\omega)|\leq \frac{\alpha_6}{k} \sum \limits_{i=J}^{xk-1}\pr^{\hat{\omega}^L}_i[\tau_{i+1}\geq M],
\end{align}
\begin{align}
  \label{deltaMM}    |\hat{I}_{J,M}(x,k,\omega)-\hat{I}_{J,M}^L(x,k,\omega)|\leq \frac{\alpha_7M}{k}\sum\limits_{i=J}^{xk-1}\sum \limits_{i=J-M}^{J+M} \xi_{i}(L-1,\omega).
\end{align}
Moreover if $J>L$ then
\begin{align*}
    \frac{1}{k} \sum \limits_{i=J}^{xk-1}\pr^{\hat{\omega}^L}_i[\tau_{i+1}\geq M]&\leq \frac{1}{k}\sum \limits_{i=J}^{xk-1}\pr^{\hat{\omega}}_i[\tau_{i+1}\geq M].
\end{align*}
\end{lemma}

\begin{proof}
For clarity, 
we suppress the arguments $k$ and $\omega$ in $\hat I$'s. 
We start by proving  \eqref{deltaLM}. Consider 
\[
\lambda^*_{M,L}(x)=\argmax_{\lambda \leq 0} \Big\{\lambda -\frac{1}{k}\sum \limits_{i=J}^{xk-1} \log \hat\phi_L^M(\lambda, \theta^i\omega) \Big\}.
\]
As in the proof of Lemma \ref{lemmaPhat}, we have that $\lambda^*_{M,L}\in (-\infty,0)$ is well defined.
By definition, 
$\hat{I}_J^L(x) \leq \hat{I}_{J,M}^L(x)$ for every $x$. For the other inequality observe that
\begin{align*}
    \nonumber\hat{I}^{L}_J(x)&=\sup \limits_{\lambda \leq 0} \Big\{\lambda -\frac{1}{k}\sum \limits_{i=J}^{xk-1} \log \hat\phi_L(\lambda, \theta^i\omega)\Big\}\\
    \nonumber&=\sup \limits_{\lambda \leq 0} \Big\{\lambda -\frac{1}{k}\sum \limits_{i=J}^{xk-1} \log [\hat\phi_L^M(\lambda, \theta^i\omega)+\ex^{\hat{\omega}^L}_ie^{\lambda \tau_{i+1}}\ind[\tau_{i+1}\geq M]\Big\}\\
    \nonumber&\geq \sup \limits_{\lambda \leq 0} \Big\{\lambda -\frac{1}{k}\sum \limits_{i=J}^{xk-1} \log [\hat\phi_L^M(\lambda, \theta^i\omega)+e^{\lambda M}\pr^{\hat{\omega}^L}_i[\tau_{i+1}\geq M]\Big\}\\
    \nonumber&= \sup \limits_{\lambda \leq 0} \Big\{\lambda -\frac{1}{k}\sum \limits_{i=J}^{xk-1} \log \hat\phi_L^M(\lambda, \theta^i\omega)-\\
    \nonumber & \qquad \qquad\frac{1}{k}\sum \limits_{i=J}^{xk-1} \log \Big[1+\frac{e^{\lambda M}\pr^{\hat{\omega}^L}_i[\tau_{i+1}\geq M]}{\hat\phi_L^M(\lambda, \theta^i\omega)}\Big]\Big\}\\
    \nonumber &\geq  \lambda^*_{M,L}(x) -\frac{1}{k}\sum \limits_{i=J}^{xk-1} \log \hat\phi_L^M(\lambda^*_{M,L}(x), \theta^i\omega)\\
    &\nonumber 
    \qquad\qquad-\frac{1}{k}\sum \limits_{i=J}^{xk-1} \log \Big[1+\frac{e^{\lambda^*_{M,L}(x) M}\pr^{\hat{\omega}^L}_i[\tau_{i+1}\geq M]}{\hat\phi_L^M(\lambda^*_{M,L}(x), \theta^i\omega)}\Big]
  \end{align*}
  and therefore
  \begin{align}
    \nonumber\hat{I}^{L}_J(x)&
    \geq   \hat{I}^{L,M}_J(x)-\frac{1}{k}\sum \limits_{i=J}^{xk-1} \log \Big[1+\frac{e^{\lambda^*_{M,L}(x) M}\pr^{\hat{\omega}^L}_i[\tau_{i+1}\geq M]}{\hat\phi_L^M(\lambda^*_{M,L}(x), \theta^i\omega)}\Big]\\
    \nonumber&\geq \hat{I}^{L,M}_J(x)-\frac{1}{k}\sum \limits_{i=J}^{xk-1} \frac{e^{\lambda^*_{M,L}(x) M}\pr^{\hat{\omega}^L}_i[\tau_{i+1}\geq M]}{\hat\phi_L^M(\lambda^*_{M,L}(x), \theta^i\omega)}\\
    \label{Lboundfinal} &\geq \hat{I}^{L,M}_J(x)-\frac{1}{k} \kappa^{-1}e^{\lambda^*_{M,L}(x) (M-1)}\sum \limits_{i=J}^{xk-1}\pr^{\hat{\omega}^L}_i[\tau_{i+1}\geq M],
\end{align}
where in \eqref{Lboundfinal} we used that
$\hat\phi_L^M(\lambda, \theta^i\omega)\geq \omega_i e^{\lambda}$. Therefore there is a constant $\alpha_6$ such that \eqref{deltaLM} holds. The proof of \eqref{deltaPM} is analogous.

Now to prove \eqref{deltaMM} we couple the random walks $X_t$ in the environment
$\hat{\omega}$ and $Y_t^L$ in the environment $\hat{\omega}^L$ as follows.
The walks  start together ($X_0=Y_0^L=J$) and move independently if they are on different sites. If $X_t=Y_t^L$  and $X_{t+1}=X_t+1$ then $Y_{t+1}^L=X_{t+1}$.
The coupling is possible because the transformed environment satisfies $\hat{\omega}^L_x\geq \hat{\omega}_x$ for every $x\geq L$. Let the joint measure of the coupling with the random walks initially at $J$ be $\hat{\mu}_J^{\hat{\omega},\hat{\omega}^L}$ and consider the event $B^{\omega,L}_{M}=\{X_t=Y_t^L, \mbox{ for all }t\in [0,M] \}$. From the definition of $\hat\omega_i^L$, using \eqref{Sdecomp} we have that if $i>L$
\begin{equation}
    \label{eq-TL2} \hat\omega^L_i\triangleq  \frac{\omega_iS(L+1,\omega^{i-L}\omega)}{S(L,\omega^{i-L}\omega)}=\omega_i\Big(\frac{S(i+1, \omega)-S(i-L, \omega)}{S(i, \omega)-S(i-L, \omega)} \Big).
\end{equation}
Observe that if $X_t=Y_t^L=i$, then the probability that they jump to different sites is $\hat{\omega}^L_i-\hat{\omega}_i$, hence it holds for $M+L< J$ that
\begin{align}
\nonumber\hat{\mu}_i^{\hat{\omega},\hat{\omega}^L}\Big((B^{\omega,L}_{M})^\complement\Big)&\leq M \max \limits_{J-M\leq i \leq J+M } (\hat{\omega}^L_i-\hat{\omega}_i)\\
\nonumber&\leq M \sum \limits_{i=J-M}^{J+M} (\hat{\omega}^L_i-\hat{\omega}_i)\\
\label{eqxL1}&= M \sum \limits_{i=J-M}^{J+M} \omega_i\frac{S(i-L, \omega)}{S(i, \omega)}\Big(\frac{S(i+1, \omega)-S(i, \omega)}{S(i, \omega)-S(i-L, \omega)}\Big).
\end{align}
By definition we have that
$S(i, \omega)=S(i-L, \omega)+e^{V_{\omega}(i-L)}S(L, \theta^{i-L}\omega)$, and therefore the right hand side of \eqref{eqxL1} equals
\begin{align*}
    &M \sum \limits_{i=J-M}^{J+M} \omega_i\frac{S(i-L, \omega)}{S(i, \omega)}\Big(\frac{e^{V_{\omega}(i)}}{e^{V_{\omega}(i-L)}S(L, \theta^{i-L}\omega)}\Big)\\
    &\leq M \sum \limits_{i=J-M}^{J+M} \frac{W(L,\theta^{i-L}\omega)}{1+W(i-L,\omega)S(L, \theta^{i-L}\omega) }\\
    &=M \sum \limits_{i=J-M}^{J+M} \xi_{i}(L-1,\omega)
    =H^\omega_J(L,M).
\end{align*}

Due to the coupling we have $\hat{\phi}_L^M(\lambda,\theta^i\omega)\geq\phi_{i,M}(\lambda,\hat{\omega})$. To get an inequality in the other
 direction
 we use that
\begin{align*}
  \nonumber\hat{\phi}^M_L(\lambda,\theta^i\omega)&=\int e^{\lambda \tau_{i+1}}\ind[\tau_{i+1}<M]d\pr^{\hat{\omega}^L}_i \\
  \nonumber&=\int e^{\lambda \tau_{i+1}}\ind[\tau_{i+1}<M]d\hat{\mu}_i^{\hat{\omega},\hat{\omega}^L} \\
  \nonumber&=\int e^{\lambda \tau_{i+1}}\ind[\tau_{i+1}<M,B^{\omega,L}_{M}]d\hat{\mu}_i^{\hat{\omega},\hat{\omega}^L}\\
  &\qquad\qquad \qquad +\int e^{\lambda \tau_{i+1}}\ind[\tau_{i+1}<M,(B^{\omega,L}_{M})^\complement]d\hat{\mu}_i^{\hat{\omega},\hat{\omega}^L} \\
   \nonumber&\leq \int e^{\lambda \tau_{i+1}}\ind[\tau_{i+1}<M]d\pr^{\hat{\omega}}_i+\hat{\mu}_i^{\hat{\omega},\hat{\omega}^L}[(B^{\omega,L}_{M})^\complement]\\
   &= \phi_{i,M}(\lambda,\hat{\omega})+\hat{\mu}_i^{\hat{\omega},\hat{\omega}^L}[(B^{\omega,L}_{M})^\complement],
\end{align*}
and thus 
\begin{equation}
\label{eqphiL}0\leq \hat{\phi}_L^M(\lambda,\theta^i\omega)-\phi_{i,M}(\lambda,\hat{\omega})\leq H^\omega_i(L,M).
\end{equation}
It is easy to check that $\hat{I}_{J,M}(x)\geq \hat{I}^{L,M}_J(x)$. Using \eqref{eqphiL} we get
 \begin{align*}
\hat{I}^{L,M}_J(x)&\geq  \sup \limits_{\lambda \leq 0 } \Big\{\lambda -\frac{1}{k}\sum\limits_{i=J}^{xk-1} \log[{\phi}_{i,M}(\lambda,\theta^i\omega)+H^\omega_i(L,M)]\Big\}\\
&=\sup \limits_{\lambda \leq 0 } \Big\{\lambda -\frac{1}{k}\sum\limits_{i=J}^{xk-1} \log{\phi}_{i,M}(\lambda,\theta^i\omega)-\frac{1}{k}\sum\limits_{i=J}^{xk-1}\log\Big[1+\frac{H^\omega_i(L,M)}{{\phi}_{i,M}(\lambda,\theta^i\omega)}\Big]\Big\}\\
&\geq \lambda^*_{M}(x) -\frac{1}{k}\sum\limits_{i=J}^{xk-1} \log{\phi}_{i,M}(\lambda^*_{M}(x),\theta^i\omega)\\
&\qquad\qquad\qquad\qquad\qquad \qquad 
-\frac{1}{k}\sum\limits_{i=J}^{xk-1}\log\Big[1+\frac{H^\omega_i(L,M)}{{\phi}_{i,M}(\lambda^*_{M}(x),\theta^i\omega)}\Big]\\
&\geq\hat{I}_{J,M}(x)-\frac{1}{k}\sum\limits_{i=J}^{xk-1}\frac{H^\omega_i(L,M)}{{\phi}_{i,M}(\lambda^*_{M}(x),\theta^i\omega)}\\
&\geq\hat{I}_{J,M}(x)-\frac{1}{k \kappa}e^{-\lambda^*_{M}(x)}\sum\limits_{i=J}^{xk-1}H^\omega_i(L,M).
\end{align*}
Using $\eqref{lambdabound}$, we conclude that 
for some constant $\alpha_7$, 
  \begin{equation}
    \label{Bo3}|\hat{I}^{L,M}_J(x)-\hat{I}_{J,M}(x)|\leq \frac{\alpha_7}{k}\sum\limits_{i=J}^{xk-1}H^\omega_i(L,M). 
  \end{equation}
The last statement of the lemma is a consequence of the inequality 
$\hat\omega^L_i\geq \hat\omega_i$  for all $i>L$.
\end{proof}
\begin{lemma}\label{Jto1} For $J>L$ we have 
\[
\left|\hat I_J^L(x,k,\omega)-\hat I_0^L(x,k,\omega) \right|\leq \frac{J}{k} \log \Big(\frac{x}{\kappa^2(1-x)}\Big).
\]
\end{lemma}
\begin{proof}
Since the variational problems only involve $\lambda\leq 0$, it is straightforward that $I_0^L(x,k,\omega)\geq I_J^L(x,k,\omega)$ if $J>L$. Now define
\[
\lambda^*_{L}(x)=\argmax \limits_{\lambda \leq 0} \Big\{\lambda -\frac{1}{k}\sum \limits_{i=0}^{xk-1} \log \hat\phi_L(\lambda, \theta^i\omega)\Big\},
\]
then
\begin{align*}
    I_J^L(x,k,\omega)&=\sup \limits_{\lambda \leq 0} \Big\{\lambda -\frac{1}{k}\sum \limits_{i=J}^{xk-1} \log \hat\phi_L(\lambda, \theta^i\omega)\Big\}\\
    &\geq \lambda^*_{L}(x) -\frac{1}{k}\sum \limits_{i=J}^{xk-1} \log \hat\phi_L(\lambda^*_{L}(x), \theta^i\omega)\\
    &=I_0^L(x,k,\omega)+\frac{1}{k}\sum \limits_{i=0}^{J-1} \log \hat\phi_L(\lambda^*_{L}(x), \theta^i\omega)\\
    &\geq I_0^L(x,k,\omega)+\frac{1}{k}\sum \limits_{i=0}^{J-1} \log (\hat\omega^L_ie^{\lambda^*_{L}(x)})\\
    &\geq I_0^L(x,k,\omega)+\frac{J}{k} \log \Big(\frac{\kappa^2(1-x)}{x}\Big),
\end{align*}
where in the last line it was used that 
$e^{\lambda^*_{L}(x)}\geq \kappa(1-x)/x$, which holds due to the same computation as \eqref{lambdabound}. 
\end{proof}
Recall the notation $\bar \xi(i,\omega)$, see \eqref{defxibar}.

\begin{lemma}\label{lemmasigma} For any measure $\eta\in M_1^s(\Sigma)$ 
  it holds that
\[
\lim \limits_{L\to \infty}\int \bar\xi(L,\omega)\eta(d\omega)=0. 
\]
\end{lemma}
\begin{proof}
We first prove the result for $\eta\in M_1^e(\Sigma)$ and then extend it 
to all stationary measures. Consider first the case
$\int \log \rho_0 \eta(d\omega)\leq 0$. Then,
\begin{align*}
  \int \bar\xi(L,\omega) \eta(d\omega)&\leq  \int W(L+1,\omega) \eta(d\omega)\\
  &=\int W(L+1,\omega) \eta(d\omega)\to 0, \mbox{ as } L\to \infty, 
\end{align*}
where the limit is due to \eqref{Wintegral}. 

Consider next the case $\mu=\int \log \rho_0 \eta(d\omega) >0$. We have
that
\begin{align*}
   & \int \bar\xi(L,\omega) \eta(d\omega)\leq C_\kappa\int \Big(1+\frac{S(L,\omega)}{S(-\infty,\omega)}\Big)^{-1} \eta(d\omega)\\
    &\leq C_\kappa\int \big(1+e^{-\mu L/2}e^{V_{\omega}(L-1)}\big)^{-1} \eta(d\omega)+\eta(S(-\infty,\cdot)>e^{\mu L/2}). 
\end{align*}
Now \eqref{boundSminus} states that
\begin{align}
    \label{Ntozero}\eta(S(-\infty,\cdot)>e^{\mu L/2})&\leq \eta\Big(N_2(\omega)>\alpha_5+\frac{\log(e^{\mu/2 L}-\alpha_4)}{\log C_\kappa}\Big).
\end{align}
Using \eqref{Wbound2} with $L$ instead of $a\log M$ we get
\begin{align}
    \label{Ntozero2}\int \big(1+e^{-\mu/2 L}e^{V_{\omega}(L-1)}\big)^{-1} \eta(d\omega)&\leq \frac{1}{1+e^{\mu L/2}}+\eta(N_1(\cdot)>L-1).
\end{align}
Since the right hand sides of both \eqref{Ntozero} and \eqref{Ntozero2} 
tend to $0$ as $L\to \infty$, the result for ergodic measures is proved. 

Now consider $\eta\in M_1^s(\Sigma)$, then there is a family of measures $\{\eta_{\theta}\}_{\theta\in\mathbb{R}}\subset M_1^e(\Sigma)$ and a measure $\nu\in M_1(\mathbb{R})$ such that
\begin{equation}
    \label{decomp} \eta=\int \eta_\theta \nu(d\theta).
\end{equation}
Using \eqref{decomp} we obtain
\begin{align}
    \label{etos1}\lim \limits_{L\to \infty}\int \bar\xi(L,\omega)\eta(d\omega)&=\lim \limits_{L\to \infty}\int \int \bar\xi(L,\omega) \nu(d\theta)d\eta_\theta(\omega),
\end{align}
and since $0\leq \bar\xi(L,\omega)\leq C_{\kappa}$ due to Lemma~\ref{Wbound}, we can apply Fubini's Theorem and dominated convergence theorem in \eqref{etos1} to obtain 
\begin{align*}
    \lim \limits_{L\to \infty}\int \bar\xi(L,\omega)\eta(d\omega)&=\lim \limits_{L\to \infty}\int \int \bar\xi(L,\omega) d\eta_\theta(\omega)\nu(d\theta)\\
    &=\int \big(\lim \limits_{L\to \infty}\int \bar\xi(L,\omega) d\eta_\theta(\omega)\Big)\nu(d\theta)\\
    &=0,
\end{align*}
which extends the result to stationary measures. 
\end{proof}

We next prove Lemma \ref{lemma2}.
\begin{proof}[Proof of Lemma~\ref{lemma2}]
We first prove that for any fixed $i$ and $L$ we have 
that 
$\hat\omega_i^{L+1}<\hat\omega_i^{L}$. From the definition of $\hat\omega^L$,
\begin{align}
    \label{Lmono1} \hat\omega_i^{L+1}-\hat\omega_i^{L}&=\frac{\omega_iS(L+2,\theta^{i-L-1}\omega)}{S(L+1,\theta^{i-L-1}\omega)}-\frac{\omega_iS(L+1,\theta^{i-L}\omega)}{S(L,\theta^{i-L}\omega)},
\end{align}
while from \eqref{Sdecomp} we have that
\begin{align}
    \nonumber S(L+2,\theta^{i-L-1}\omega)&=S(1,\theta^{i-L-1}\omega)+\rho_1(\theta^{i-L-1}\omega)S(L+1,\theta^{i-L}\omega)\\
    \label{Lmono2}&=1+\rho_{i-L}S(L+1,\theta^{i-L}\omega),
\end{align}
and analogously
\begin{align}
    \label{Lmono3}S(L+1,\theta^{i-L-1}\omega)&=1+\rho_{i-L}S(L,\theta^{i-L}\omega).
\end{align}
Using \eqref{Lmono2} and \eqref{Lmono3} in \eqref{Lmono1} we obtain
\begin{align*}
    \hat\omega_i^{L+1}-\hat\omega_i^{L}&=\omega_i\Big(\frac{1+\rho_{i-L}S(L+1,\theta^{i-L}\omega)}{1+\rho_{i-L}S(L,\theta^{i-L}\omega)}-\frac{S(L+1,\theta^{i-L}\omega)}{S(L,\theta^{i-L}\omega)}\Big)\\
    &=\frac{\omega_i(S(L,\theta^{i-L}\omega)-S(L+1,\theta^{i-L}\omega))}{(1+\rho_{i-L}S(L,\theta^{i-L}\omega))S(L,\theta^{i-L}\omega)}<0,
\end{align*}
which proves the claimed monotonicity.
With it we can define a coupling analogous to the one defined in the proof of Lemma~\ref{lemmaapproxPhat}: Take $L'>L$, so for every $i$ we have that
$\hat\omega_i^{L'}<\hat\omega_i^{L}$, hence we can define two random walks $Y^{L'}_t$ and $Y^L_t$ such that for all $t$ we have that
$Y^{L'}_t\leq Y^L_t$. 
This in turns implies that 
$\hat{\phi}_L(\lambda, \omega)$ is non-increasing in $L$ for $\lambda\leq 0$,
which in turns gives that
\begin{equation}
 \label{Iphidef}   I_L^{\phi}(x,\eta):=\sup\limits_{\lambda \leq 0}\Big\{\lambda -x\int \log \phi(\hat\omega^L,\lambda)\eta(d\omega) \Big\}
\end{equation}
is increasing in $L$. This yields the claimed existence of the limit in 
\eqref{eq-IL}.
\end{proof}

\section{Identification of the rate function,
  study of the 
  variational problem, and proof of Proposition~\ref{limitteo}}
\label{sec-varprob}
In this section, we prove Proposition \ref{limitteo}.
We begin with a preliminary computation.
\begin{lemma}\label{lemmamins} With $\Lambda(\lambda)=\log\int \rho_0^\lambda d\eta$, set
\begin{equation}
    \label{Imdef}I_m(x)=\sup \limits_{\lambda \in \mathbb{R}} \{\lambda x - \Lambda(\lambda) \},
\end{equation}
then, with $\eta$ such that $s\in (1,\infty]$ as in the statement of Theorem \ref{mainteo},
\[
\inf \limits_{z>0}\frac{I_m(z)}{z}=s.
\]
\end{lemma}

\begin{proof}
  The claim is trivial if $s=\infty$, since then $I_m(z)=\infty$ for any $z>0$.
  So we restrict attention to $s\in (1,\infty)$.
Since $\Lambda(\cdot)$ is differentiable, 
we can compute the critical point $\tilde\lambda(x)$ of $G_x(\lambda)=\lambda x - \Lambda(\lambda)$ to get
\begin{align}
    \label{eq-Gx}G'_x(\tilde\lambda(x))=x-\Lambda'(\tilde\lambda(x))=0,
\end{align}
and then
\[
\inf \limits_{z>0}\frac{I_m(z)}{z}=\inf \limits_{z>0}\frac{\tilde\lambda(z)z-\Lambda(\tilde\lambda(z))}{z}=\inf \limits_{z>0}\Big(\tilde\lambda(z)-\frac{\Lambda(\tilde\lambda(z))}{z}\Big)
=:\inf\limits_{z>0} H(z).
\]
Since $\Lambda(\cdot)$ is analytic and (strictly) convex, its second derivative is
strictly positive and so, 
by \eqref{eq-Gx} and the implicit function theorem,
the map $z\mapsto \tilde\lambda(z)$ is differentiable.
Thus, the function $z\mapsto H(z)$ is differentiable on $\mathbb{R}_+$ and its derivative is
\[
H'(z)=\tilde\lambda'(z)-\frac{\Lambda'(\tilde\lambda(z))\tilde\lambda'(z)}{z}+\frac{\Lambda(\tilde\lambda(z))}{z^2}=\frac{\Lambda(\tilde\lambda(z))}{z^2}.
\]
Therefore, 
the only positive solution to $H'(z)=0$ is the solution to
$\Lambda(\tilde\lambda(z))=0$. Thus,
for  $z^*$ satisfying the latter equality we have 
\[
\inf \limits_{z>0} H(z)=\tilde\lambda(z^*)-\frac{\Lambda(\tilde\lambda(z^*))}{z^*}=\tilde\lambda(z^*).
\]
This means that the value of $\inf \limits_{z>0}{I_m(z)}/{z}$ is the solution 
to the equation
\[
\log  \int\rho_0^\lambda d\eta =0,
\]
concluding the proof. 
\end{proof}

As a second step toward the proof of Proposition \ref{limitteo}, we derive
upper and lower bounds on the scaled logarithmic limit of $\chi$.
Recall the definition of the quantities $f$ and $g$ in \eqref{fgdef} and $\chi$ in $\eqref{eq-chi}$. From now on for simplicity we denote the vectors
of measures $\{\eta\}^\e_{x}=(\eta_{1},\ldots, \eta_{x/\e})$ and $\{\bar\eta\}^\e_{x,c}=(\bar{\eta}_{1},\ldots, \bar{\eta}_{\lfloor(c-x)/\e\rfloor})$.
Recall the notation $\chi(k,x,c,\eta)$ and $\Delta_\e$, see \eqref{eq-chi}
and \eqref{Deltadef}, and  the constant $C_\kappa$, see \eqref{eq-ckappa}.

\begin{lemma} \label{lemmaLU} Under the assumptions of Theorem~\ref{mainteo}, we have that
\begin{equation}
 \label{LUchi}   I^*_l(x,c) \leq \lim \limits_{k\to \infty} -\frac{1}{k}\log \chi(k,x,c) \leq  I^*_u(x,c),
\end{equation}
where
\begin{align}
    \label{Ifore}I^*_u(x,c)&=\inf \limits_{\{\eta\}^\e_{x},\{\bar\eta\}^\e_{x,c}\in M_1^s(\Sigma_p)} \Big\{ \Big[\Big(\Delta_\e (\{\eta\}^\e_{x},\{\bar\eta\}^\e_{x,c})-\e\log C_{\kappa} \Big)^+\\
    \nonumber&-{I}^{f}\Big(x,\frac{\e}{x} \sum \limits_{i=1}^{x/\e}\eta_i\Big)\Big]^-+{\e}\sum_{i=1}^{x/\e}h(\eta_i|p)+{\e}\sum_{j=1}^{\lfloor(c-x)/\e\rfloor}h(\bar{\eta}_j|p) \Big\},
\end{align}
and 
\begin{align}
    \label{Ilower}I^*_l(x,c)&=\inf \limits_{\{\eta\}^\e_{x},\{\bar\eta\}^\e_{x,c}\in M_1^s(\Sigma_p)} \Big\{ \Big[\Big(\Delta_\e (\{\eta\}^\e_{x},\{\bar\eta\}^\e_{x,c})+\e\log C_{\kappa} \Big)^+\\
    \nonumber&-{I}^{f}\Big(x,\frac{\e}{x} \sum \limits_{i=1}^{x/\e}\eta_i\Big)\Big]^-+{\e}\sum_{i=1}^{x/\e}h(\eta_i|p)+{\e}\sum_{j=1}^{\lfloor(c-x)/\e\rfloor}h(\bar{\eta}_j|p) \Big\}.
\end{align}
\end{lemma}

\begin{proof}
The strategy of the proof is as follows.
We first use
Lemmas~\ref{lemmaPhat} and ~\ref{Jto1} to 
express 
the functions $f$ and $g$ as 
exponentials of functionals of empirical fields plus 
deterministic error terms, and use these expressions
in upper and lower 
bounds on $\chi$ (see \eqref{upperchi} and \eqref{lowerchi}). 
Varadhan's Lemma is then invoked to compute the re-scaled logarithmic limit 
in terms of a variational problem (see \eqref{eq-finalvar}).

From now on, we denote $\hat I^L_0(x,k,\omega)=I^L(x,R_{xk})$, 
where $R_{xk}$ is as in \eqref{eq-Rn} with $n=xk$, and throughout $L,M,J$ are
positive integers as in Section \ref{sec-lem13}. In what follows, we will 
assume that $L\leq xk$ and that $J\geq L+M$. (Eventually, we will take limits
in $k\to\infty$, followed by $J\to\infty$ and then by $M,L\to \infty$.)

We begin by estimating $f(\omega, x,k)$:
\begin{align}
 \nonumber -\frac{1}{k}\log (1-f(\omega,x,k))&=-\frac{1}{k}\log \pr_1^{\hat\omega}[\tau_{xk}\leq k]\\
 &\geq-\frac{1}{k}\log \pr_J^{\hat\omega}[\tau_{xk}\leq k]
  \label{ineq3} \geq \hat{I}_J(x,k,\omega),
\end{align}
where in \eqref{ineq3} we used Lemma~\ref{lemmaPhat}.
Since for $i> L$ we have that
$\hat{\omega}_i^L\geq \hat{\omega}_i$, for $J>L$ we deduce that
$\hat{I}_J(x,k,\omega)\geq \hat{I}_J^L(x,k,\omega)$. Using Lemma~\ref{Jto1} 
it follows that for some positive constant $\alpha_8$ independent of $k$ and $J$,
\begin{align}
\nonumber-\frac{1}{k}\log (1-f(\omega,x,k))&=-\frac{1}{k}\log \pr_1^{\hat\omega}[\tau_{xk}\leq k]\\
\label{fLbound} &\geq\Big(\hat I^L(x,R_{xk})-\frac{\alpha_8J}{k}\Big)^+.
\end{align}
For the upper bound, let $x'_l=x(1+l)$ and $k'_l=k/(1+l)$. Then,
\begin{align}
\nonumber -\frac{1}{k}\log (1-f(\omega,x,k))&= -\frac{1}{k}\log\pr_1^{\hat\omega}[\tau_{xk}\leq k]\\
\nonumber &=-\frac{1}{k}\log\pr_1^{\hat\omega}[\tau_{x'_lk'_l}\leq k'_l(1+l)]\\
\label{ineq4}&\leq \hat{I}_{1,M}(x'_l,k'_l,\omega)+\alpha_3l-\frac{1}{k'_l}\log\big(1-2e^{-\frac{2k'_ll^2}{x'_lM^2}}\big),
\end{align}
where in \eqref{ineq4} again we used Lemma~\ref{lemmaPhat}. By Lemma~\ref{lemmaapproxPhat} we have 
\begin{align*}
  \hat{I}_{1,M}(x'_l,k'_l,\omega)&\leq \hat{I}_1^L(x'_l,k'_l,\omega)+\frac{\alpha_7M}{k}\sum\limits_{i=J}^{xk-1}\sum \limits_{i=J-M}^{J+M} \xi_{i}(L-1,\omega)\\
    &\qquad+\frac{\alpha_6}{k} \sum \limits_{i=J}^{xk-1}\pr^{\hat{\omega}^L}_i[\tau_{i+1}\geq M].
\end{align*}
Using Lemma~\ref{lemmaptau} and \eqref{Hbound}, we obtain
\begin{align}
  \nonumber\hat{I}_{1,M}(x'_l,k'_l,\omega)&\leq\hat{I}_1^L(x'_l,k'_l,\omega)+\alpha_6x'_l\Phi_{1}(M,R_{xk})\\
    \label{pbound2}&\qquad+  \alpha_7x'_l\Phi_{2,M}(L,R_{xk})+\alpha_6x'_le^{-M^{1+a\log \kappa}/(a\log M)},
\end{align}
where we used the abbreviations
\begin{equation}
   \label{def1}\Phi_{1}(M,R_{xk})=\int\bar\xi(a\log M,\omega)R_{xk}(d\omega), 
\end{equation}
\begin{equation}
    \label{def2}\Phi_{2,M}(L,R_{xk})=M\sum_{j=-M}^{M} \int \bar\xi(L-1,\theta^{j-M-L+1}\omega)R_{xk}(d\omega).
\end{equation}
Using \eqref{pbound2} in \eqref{ineq4} we obtain
\begin{align}
  \label{fUbound}
    &\frac{1}{k}\log (1-f(\omega,x,k))\\
  \nonumber  &\leq \hat{I}^L(x'_l,R_{xk})+\alpha_6x'_l\Phi_{1}(M,R_{xk})+\alpha_7x'_l \Phi_{2,M}(L,R_{xk})+\beta(k'_l,l,M,L)\\
    \nonumber&=:\hat{I}^{\Phi}_{l,M,L}(x,R_{xk})+\beta(k'_l,l,M,L),
\end{align}
where again for clarity we denoted
\begin{equation}
    \label{def3}\beta(k,l,M,L)=\alpha_3l-\log \big(1-2e^{-\frac{2kl^2}{x'_lM^2}}\big)+\alpha_6x'_le^{-M^{1+a\log \kappa}/(a\log M)}+\alpha_8J/k.
\end{equation}
Turning to bound $g(\omega,x,c,k)$ in \eqref{eq-chi}, we define
\begin{align*}
I_l^{g,\e}(\omega,x,c)=\Big(\Delta_{\e}-\e\log C_{\kappa} \Big)^+,\quad  I_u^{g,\e}(\omega,x,c)=\Big(\Delta_{\e}+\e\log C_{\kappa} \Big)^+.
\end{align*}
Lemma~\ref{lemma1} and the definition of $g$, see \eqref{fgdef},
imply that
\begin{equation}
\label{gbounds1}g(\omega,x,c,k)\geq \exp \{-k I_u^{g,\e}(\omega,x,c)-{(c+x)\log C_\kappa}/{\e}-\log(xk)\},
\end{equation}
and 
\begin{equation}
\label{gbounds2}g(\omega,x,c,k)\leq \exp \{-k I_l^{g,\e}(\omega,x,c)+{(c+x)\log C_\kappa}/{\e}+\log (ck)\}.
\end{equation}
We now use the bounds on $f,g$ to estimate $\chi$. Denote 
\begin{equation}
    \label{errorq1} q^{c,\e}_{1,k}={(c+x)\log C_\kappa}/{\e}+\log(xk).
\end{equation}
We use \eqref{fLbound} and \eqref{gbounds1} to get an upper bound for $\chi  (k,x,c)$, as follows. 
\begin{align}
  \label{upperchi}
&\chi  (k,x,c)=\int \Big(\frac{1-f(\omega,x,k)}{1-f(\omega,x,k)(1-g(\omega,x,c,k))} \Big) P(d\omega)
\\ \nonumber &\leq  \int \frac{e^{-(kI^L(x,R_{xk})-\alpha_8J)^+}}{1-(1- e^{-(kI^L(x,R_{xk})-\alpha_8J)^+})(1- e^{-k I_u^{g,\e}(\omega,x,c)-q^{c,\e}_{1,k}})}P(d\omega)\\
\nonumber&=\int \Big(1+\exp\Big\{(kI^L(x,R_{xk})-\alpha_8J)^+-k I_u^{g,\e}(\omega,x,c)-q^{c,\e}_{1,k}\Big\}\\
\nonumber&\qquad \qquad - \exp\Big\{-k I_u^{g,\e}(\omega,x,c)-q^{c,\e}_{1,k}\Big\} \Big)^{-1}P(d\omega)\\
\nonumber&\stackrel{(a)}{\leq} \int \exp\Big\{-\Big(k I_u^{g,\e}(\omega,x,c)+q^{c,\e}_{1,k}-(kI^L(x,R_{xk})-\alpha_8J)^+\Big)^-\Big\} P(d\omega)\\
\nonumber&\leq \int \exp\Big\{-\Big(k I_u^{g,\e}(\omega,x,c)+q^{c,\e}_{1,k}+\alpha_8J-kI^L(x,R_{xk})\Big)^-\Big\} P(d\omega)\\
&\leq e^{q^{c,\e}_{1,k}+\alpha_8J} \int \exp\Big\{-\Big(k I_u^{g,\e}(\omega,x,c)-kI^L(x,R_{xk})\Big)^-\Big\} P(d\omega),
\nonumber\end{align}
where in $(a)$  we used that for $a,b\geq 0$ we have
that $(1+e^{a-b}-e^{-b})^{-1}\leq e^{-(b-a)^-}$. Denote
\begin{equation}
    \label{errorq2}q^{c,\e}_{2,k}={(c+x)\log C_\kappa}/{\e}+\log (ck).
\end{equation}
Similarly we use \eqref{fUbound} and \eqref{gbounds2} to get the lower bound. 
\begin{align}
  \label{lowerchi}
&\chi  (k,x,c)=\int \Big(\frac{1-f(\omega,x,k)}{1-f(\omega,x,k)(1-g(\omega,x,c,k))} \Big) P(d\omega)\\ \nonumber
&\geq \int \frac{e^{-k\hat{I}^{\Phi}_{l,M,L}(x,R_{xk})-k\beta(k'_l,l,M,L)}}{1-(1- e^{-k\hat{I}^{\Phi}_{l,M,L}(x,R_{xk})-k\beta(k'_l,l,M,L)})(1- e^{-k I_l^{g,\e}(\omega,x,c)+q^{c,\e}_{2,k}})}P(d\omega)\\
\nonumber&= \int \Big(1+\exp\Big\{k\hat{I}^{\Phi}_{l,M,L}(x,R_{xk})+k\beta(k'_l,l,M,L)-k I_l^{g,\e}(\omega,x,c)+q^{c,\e}_{2,k}\Big\}\\
\nonumber&\qquad\qquad- \exp\Big\{-k I_l^{g,\e}(\omega,x,c)+q^{c,\e}_{2,k}\Big\} \Big)^{-1}P(d\omega)\\
\nonumber&\geq  \int \Big(1+\exp\Big\{k\hat{I}^{\Phi}_{l,M,L}(x,R_{xk})+k\beta(k'_l,l,M,L)-k I_l^{g,\e}(\omega,x,c)+q^{c,\e}_{2,k}\Big\}\Big)^{-1}P(d\omega)\\
\nonumber&\stackrel{(b)}{\geq} \frac{1}{2}\int \exp\Big\{-\Big(k I_l^{g,\e}(\omega,x,c)-q^{c,\e}_{2,k}-k\hat{I}^{\Phi}_{l,M,L}(x,R_{xk})-k\beta(k'_l,l,M,L)\Big)^-\Big\} P(d\omega)\\
&\geq \frac{e^{-q^{c,\e}_{2,k}-k\beta(k'_l,l,M,L)}}{2}\int \exp\Big\{-\Big(k I_l^{g,\e}(\omega,x,c)-k\hat{I}^{\Phi}_{l,M,L}(x,R_{xk})\Big)^-\Big\} P(d\omega),
\nonumber\end{align}
where in $(b)$ we used that $(1+e^a)^{-1}\geq e^{-a^+}/2$ and $(a-b)^+=(b-a)^-$.

We next take the rescaled logarithmic limit of $\chi$.
This has do be done in both the bounds of \eqref{upperchi} and \eqref{lowerchi}, but since the proofs are similar we only consider the later, i.e. \eqref{lowerchi}, which is slightly more complex.  
Taking the rescaled logarithm,   \eqref{lowerchi} becomes
\begin{align}
   \label{upperchi2} -\frac{1}{k}\log \chi &\leq -\frac{1}{k} \log \int \exp\Big\{-\Big(k I_l^{g,\e}(\omega,x,c)-k\hat{I}^{\Phi}_{l,M,L}(x,R_{xk})\Big)^-\Big\} P(d\omega)\\
   \nonumber&\qquad+\frac{q^{c,\e}_{2,k}+k\beta(k'_l,l,M,L)+\log 2}{k}.
\end{align}
From \eqref{def3},
\begin{align}
     \label{errorterm1}\lim \limits_{k\to \infty}\beta(k'_l,l,M,L)&=\alpha_3l+\alpha_6x'_le^{-M^{1+a\log \kappa}/(a\log M)},
\end{align}
while
\begin{equation}
   \lim \limits_{k\to \infty} \frac{q^{c,\e}_{2,k}+\log 2}{k}=0.
\end{equation}
This controls the last term in the right-hand side \eqref{upperchi2}. The exponent in the integral in the right-hand side of \eqref{upperchi2} is non-positive, moreover the empirical fields emerge only in 
two functions in the exponent, viz.
\begin{align*}
    F_1(\eta)=\int \log \rho_0(\omega)\eta(d\omega),\quad F_2(\eta)=\int \log \phi(\lambda,\omega) \eta(d\omega).
\end{align*}
Such functions are measurable on $M_1(\Sigma^-)$. It is straightforward that the mapping $\eta\to F_1(\eta)$ is continuous in the weak topology. Moreover from
\cite[Lemma 6]{LDP} we have that $F_2(\eta)$ is also continuous in the weak topology. Recall, see Lemma~\ref{blocksLDP}, that the empirical fields $(\{ R_{i,\e}\}_{i=1,\ldots,x/\e},\{ \bar{R}_{i,\e}\}_{i=1,\ldots,\lfloor(c-x)/\e\rfloor})$ satisfy a LDP with rate function 
\[
I_B(\{\eta\}^\e_{x},\{\bar\eta\}^\e_{x,c})={\e}\sum_{i=1}^{x/\e}h(\eta_i|p)+{\e}\sum_{j=1}^{\lfloor(c-x)/\e\rfloor}h(\bar{\eta}_j|p).
\]
 We apply Varadhan's lemma to the limit of the integral term in \eqref{upperchi2} and get
 \begin{align}
   \label{eq-finalvar}
I'_u(x,c) &\triangleq \lim \limits_{k\to \infty}-\frac{1}{k}\log \int
  \exp\Big\{-\Big(k I_l^{g,\e}(\omega,x,c)-k\hat{I}^{\Phi}_{l,M,L}(x,R_{xk})\Big)^-\Big\} P(d\omega)\\
 \label{infimum2} &= \inf \limits_{\{\eta\}^\e_{x},\{\bar\eta\}^\e_{x,c}
 \in M_1^s(\Sigma_p)} \mathcal{W}(\{\eta\}^\e_{x},\{\bar\eta\}^\e_{x,c}),
 \end{align}
where
 \begin{align*}
\mathcal{W}(\{\eta\}^\e_{x},\{\bar\eta\}^\e_{x,c})&= \Big[\Big(\Delta_\e (\{\eta\}^\e_{x},\{\bar\eta\}^\e_{x,c})
-\e\log C_{\kappa} \Big)^+\!\!\!-\hat{I}^{\Phi}_{l,M,L}\Big(x,\frac{\e}{x} \sum \limits_{i=1}^{x/\e}\eta_i\Big)\Big]^-\\
&\qquad\qquad+{\e}\sum_{i=1}^{x/\e}h(\eta_i|p)+{\e}\sum_{j=1}^{\lfloor(c-x)/\e\rfloor}h(\bar{\eta}_j|p).
 \end{align*}
The next step is to replace 
$\hat{I}^{\Phi}_{l,M,L}$ with ${I}^{f}$ in the variational problem. First recall the definition of $I^*_u(x,c)$ in \eqref{Ifore}. Fixing $\e'>0$, there exist a vector $(\{\tilde\eta^{\e'}\},\{\tilde\eta^{*,\e'}\})$ such that \begin{align*}
I^*_u(x,c)&>\Big[\Big(\Delta_\e (\{\tilde\eta^{\e'}\},\{\tilde\eta^{*,\e'}\})-\e\log C_{\kappa} \Big)^+-{I}^{f}\Big(x,\frac{\e}{x} \sum \limits_{i=1}^{x/\e}\tilde\eta^{\e'}_{i}\Big)\Big]^-\\
&+{\e}\sum_{i=1}^{x/\e}h(\tilde\eta^{\e'}_{i}|p)+{\e}\sum_{j=1}^{\lfloor(c-x)/\e\rfloor}h(\tilde\eta_j^{*,\e'}|p)-\e'.
\end{align*}
From \eqref{infimum2},
\begin{align}
    \label{IforI} &I'_u(x,c)\leq \Big[\Big(\Delta_\e (\{\tilde\eta^{\e'}\},\{\tilde\eta^{*,\e'}\})
-\e\log C_{\kappa} \Big)^+\\
\nonumber &-\hat{I}^{\Phi}_{l,M,L}\Big(x,\frac{\e}{x} \sum \limits_{i=1}^{x/\e}\tilde\eta^{\e'}_{i}\Big)\Big]^-+{\e}\sum_{i=1}^{x/\e}h(\tilde\eta^{\e'}_{i}|p)+{\e}\sum_{j=1}^{\lfloor(c-x)/\e\rfloor}h(\tilde\eta_j^{*,\e'}|p).
\end{align}
Now recall, see \eqref{fUbound}, that for a fixed measure $\eta$, 
\begin{align*}
    \hat{I}^{\Phi}_{l,M,L}(x,\eta)&=\hat{I}^L(x'_l,\eta)+\alpha_6x'_l\Phi_{1}(M,\eta)+\alpha_7x'_l \Phi_{2,M}(L,\eta).
\end{align*}
From the definitions of $\Phi_1$ and $\Phi_2$ in \eqref{def1} and \eqref{def2}, using Lemma~\ref{lemmasigma} we obtain 
\begin{align*}
\lim\limits_{M\to \infty}\Phi_{1}(M,\eta)=\lim\limits_{L\to \infty}\Phi_{2,M}(L,\eta)=0,
\end{align*}
and since $\{\tilde\eta^{\e'}\}$ does not depend on the variables $l,M,L$, 
\begin{align*}
\lim\limits_{M\to \infty}\lim\limits_{L\to \infty}\hat{I}^{\Phi}_{l,M,L}\Big(x,\frac{\e}{x} \sum \limits_{i=1}^{x/\e}\tilde\eta^{\e'}_{i}\Big)={I}^{f}\Big(x'_l,\frac{\e}{x} \sum \limits_{i=1}^{x/\e}\tilde\eta^{\e'}_{i}\Big).
\end{align*}
From \eqref{eq-IL} we have $I^F$ is a point-wise limit of convex functions, hence it is convex and since it is defined on the interval $(0,1)$, it is also continuous on it. Taking the limit on $l$ we obtain
\begin{align}
\label{PhitoPhi}\lim\limits_{l \to 0}\lim\limits_{M\to \infty}\lim\limits_{L\to \infty}\hat{I}^{\Phi}_{l,M,L}\Big(x,\frac{\e}{x} \sum \limits_{i=1}^{x/\e}\tilde\eta^{\e'}_{i}\Big)={I}^{f}\Big(x,\frac{\e}{x} \sum \limits_{i=1}^{x/\e}\tilde\eta^{\e'}_{i}\Big).
\end{align}
Using \eqref{PhitoPhi} on \eqref{IforI} and comparing to \eqref{Ifore} we obtain,
\begin{align*}
    \lim\limits_{l\to 0}\lim\limits_{M\to \infty}\lim\limits_{L\to \infty}I'_u(x,c)&\leq I^*_u(x,c)+\e'.
\end{align*}
Since this holds for every $\e'>0$ and using
\[
\lim\limits_{l\to 0}\lim\limits_{M\to \infty}\alpha_3l+\alpha_6x'_le^{-M^{1+a\log \kappa}/(a\log M)}=0
\]
in \eqref{errorterm1}, it follows that
\begin{align}
     \label{chiIU} \lim \limits_{n\to \infty}& -\frac{1}{k}\log \chi \leq I^*_u(x,c).
\end{align}
Arguing similarly for the lower bound, we also have
\begin{align}
      \label{chiIL}    \lim \limits_{n\to \infty}& -\frac{1}{k}\log \chi \geq I^*_l(x,c),
\end{align}
which concludes the proof. 
\end{proof}
As a last preparatory step, we compute the minimization over measures of certain
functionals appearing in the expressions $I_u^*$ and $I_l^*$.
\begin{lemma}\label{IFlowerlemma} Define
\begin{equation}
 \label{F3def}  F_3(\{\bar\eta\}^\e_{x,c}):= \max \limits_{1\leq j\leq \lfloor(c-x)/\e\rfloor} \sum_{i=1}^j \int \log\rho_0(\omega) \bar{\eta}_i(d\omega),
\end{equation}
and let 
\begin{equation}
  \label{IFdef} I^F_c(y):= \inf\Big\{\sum_{j=1}^{\lfloor(c-x)/\e\rfloor}h(\bar{\eta}_j|p): F_3(\{\bar\eta\}^\e_{x,c})=y   \Big\}. 
\end{equation}
Then 
\begin{equation}
    \label{IFgeqsy}I^F_c(y)\geq sy.
\end{equation}
\end{lemma}

\begin{proof}
The result is trivial for $y\leq 0$ and hence we only consider $y> 0$. Also,
if $s=\infty$ then the support of $\log \rho_0$ under $p$ is contained in
$\mathbb{R}_-$ and then $I_c^F(y)=\infty$ for $y>0$. So
we consider in the sequel $s\in (1,\infty)$.
Since the only appearance
of the measures $\{\bar\eta\}^\e_{x,c}$ in $F_3(\{\bar\eta\}^\e_{x,c})$ is through integration
against the test function $ F_1(\eta)=\int \log \rho_0(\omega)\eta(d\omega)$, we consider the
auxiliary problem
\begin{equation}
\label{Imdef2}  I^*_m(x)= \inf\limits_{\eta \in M_1^s(\Sigma_p)}\Big\{h(\eta|p): \int \log \rho_0(\omega) \eta(d\omega)=x \Big\}.
\end{equation}
By the contraction principle, see \cite[Theorem 4.2.1]{DZ}, $I^*_m(x)$
 is the rate function of the LDP for 
the random variable $\int \log \rho_0(\omega) R_n(d\omega)$, under $P=p^\mathbb{Z}$,
and thus
by Cramer's theorem, we have that  the function $I^*_m(\cdot)$ coincides with
the function $I_m(\cdot)$ defined in \eqref{Imdef} when $\eta = P$.

To be able to reduce the analysis of $F_3(\{\bar\eta\}^\e_{x,c})$ to a single variable, we use \eqref{F3def} and \eqref{Imdef2} in \eqref{IFdef}, obtaining
\begin{equation}
  \label{IFdef2} I^F_c(y)=\inf\Big\{\sum_{i=1}^{\lfloor(c-x)/\e\rfloor}I_m(x_i): \max_{1\leq j \leq \lfloor(c-x)/\e\rfloor} x_1+\ldots+x_j=y \Big\}. 
\end{equation}
Now for a vector $\mathbf{x}\in \mathbb{R}^{\lfloor (c-x)/\e\rfloor}$, define
\begin{equation}
    \label{Ndef} N(\mathbf{x})=\inf\Big\{m\in \mathbb{Z}^+: \max \limits_{1\leq j\leq \lfloor (c-x)/\e\rfloor} \sum \limits_{i=1}^jx_i =\sum \limits_{i=1}^m x_i  \Big\}.
\end{equation}
Using \eqref{Ndef} we have 
\begin{align}
    \nonumber I^F_c(y)&=\inf\Big\{\sum_{i=1}^{\lfloor(c-x)/\e\rfloor}I_m(x_i): \max_{1\leq j \leq \lfloor(c-x)/\e\rfloor} x_1+\ldots+x_j=y \Big\}\\
   \nonumber &\geq \inf\Big\{\sum_{i=1}^{N(\mathbf{x})}I_m(x_i): \max_{1\leq j \leq \lfloor(c-x)/\e\rfloor} x_1+\ldots+x_j=y \Big\}\\
   \label{IFfinallower} &{\geq}\inf\Big\{N(\mathbf{x}) I_m\Big(\frac{y}{N(\mathbf{x})}\Big): \max_{1\leq j \leq \lfloor(c-x)/\e\rfloor} x_1+\ldots+x_j=y \Big\},
\end{align}
where the last inequality uses the convexity of $I_m$.
By 
Lemma~\ref{lemmamins}, for any positive $z$ we have $I_m(z)\geq sz$, and hence
\begin{equation}
    N(\mathbf{x}) I_m\Big(\frac{y}{N(\mathbf{x})}\Big)\geq sy,
\end{equation}
which together with \eqref{IFfinallower} proves \eqref{IFgeqsy}.
\end{proof}

We are finally ready to prove Proposition \ref{limitteo}.  

\begin{proof}[Proof of Proposition~\ref{limitteo}]

To prove the proposition we simplify the variational problems in Lemma~\ref{lemmaLU} using the property of positive velocity and then take the limit on $\e\to 0$ to match the lower and upper bounds. 

As in Lemma~\ref{lemmaLU} the calculations for $I^*_u(x,c)$ and $I^*_l(x,c)$ are similar, so we will only develop $I^*_l(x,c)$ as it has more technical details to consider.  
Define
\begin{align}
\label{F4def}F_4(\{\eta\}^\e_{x})=\max \limits_{1\leq j\leq x/\e} \sum_{i=1}^j \int \log\rho_0(\omega) \eta_{i}(d\omega)-\sum_{i=1}^{x/\e} \int \log\rho_0(\omega) \eta_{i}(d\omega)\geq 0.
\end{align}
Using \eqref{F3def} and \eqref{F4def}, $\Delta_\e (\{\eta\}^\e_{x},\{\bar\eta\}^\e_{x,c})$, defined in \eqref{Deltadef}, can be writen as 
\begin{equation}
\label{Delta2}\Delta_\e (\{\eta\}^\e_{x},\{\bar\eta\}^\e_{x,c})=\e F_3(\{\bar\eta\}^\e_{x,c})-\e F_4(\{\eta\}^\e_{x}).
\end{equation}
Substituting
\eqref{Delta2} in \eqref{Ilower} and using \eqref{IFdef}, $I^*_l(x,c)$ reduces to
\begin{align}
 \label{ILfinal}  I_l^*(x,c)&=\inf \limits_{\{\eta\}^\e_{x} \in M_1^s(\Sigma_p), y\in \mathbb{R}} \Big\{ \Big[ \e\Big(y - F_4(\{\eta\}^\e_{x})+\log C_{\kappa}\Big)^+-I^{F}\Big(x,\frac{\e}{x} \sum \limits_{i=1}^{x/\e}\eta_i\Big)\Big]^-\\
\nonumber&\qquad+{\e}\sum_{i=1}^{x/\e}h(\eta_i|p)+\e I^F_c(y)\Big\}.
 \end{align}
Now observe that 
 \begin{align}
\nonumber \Big[& \e\Big(y - F_4(\{\eta\}^\e_{x})+\log C_{\kappa}\Big)^+-I^{F}\Big(x,\frac{\e}{x} \sum \limits_{i=1}^{x/\e}\eta_i\Big)\Big]^-+\e\sum_{i=1}^{x/\e}h(\eta_i|p)+\e I^F_c(y)\\
\nonumber  &\stackrel{(d)}{\geq} I^{F}\Big(x,\frac{\e}{x} \sum \limits_{i=1}^{x/\e}\eta_i\Big)- \e\Big(y - F_4(\{\eta\}^\e_{x})+\log C_{\kappa}\Big)^++\e\sum_{i=1}^{x/\e}h(\eta_i|p)+\e I^F_c(y)\\
\label{byey}&\stackrel{(e)}{\geq}  I^{F}\Big(x,\frac{\e}{x} \sum \limits_{i=1}^{x/\e}\eta_i\Big)+\e\sum_{i=1}^{x/\e}h(\eta_i|p)-\e\log C_{\kappa}+\e (I^F_c(y)-y^+),
 \end{align}
where in $(d)$ we used that $[a-b]^-\geq b-a$ and in $(e)$ we used that $(a+b)^+\leq a^++b^+$ and that $F_4(\{\eta\}^\e_{x})\geq 0$. By
Lemma~\ref{IFlowerlemma} we have that
\begin{equation}
\label{boundFy} I^F_c(y)-y^+\geq 0
\end{equation}
Substituting
\eqref{boundFy} in \eqref{byey} and taking the infimum over $\{\eta\}^\e_{x} \in M_1^s(\Sigma_p)$ and $y\in \mathbb{R}$ we obtain 
\begin{equation}
\label{ILbound}    I_l^*(x,c) \geq \inf \limits_{\{\eta\}^\e_{x} \in M_1^s(\Sigma_p)} \Big\{ I^{F}\Big(x,\frac{\e}{x} \sum \limits_{i=1}^{x/\e}\eta_i\Big)+\e\sum_{i=1}^{x/\e}h(\eta_i|p)\Big\}- \e\log C_{\kappa}.
\end{equation}
Next we
deal with the upper bound. As with $I_l^*(x,c)$, we substitute
\eqref{Delta2} and \eqref{IFdef} in \eqref{Ifore} to obtain
\begin{align}
 \label{IUfinal}  I_u^*(x,c)&=\inf \limits_{\{\eta\}^\e_{x} \in M_1^s(\Sigma_p), y\in \mathbb{R}} \Big\{ \Big[ \e\Big(y - F_4(\{\eta\}^\e_{x})-\log C_{\kappa}\Big)^+-I^{F}\Big(x,\frac{\e}{x} \sum \limits_{i=1}^{x/\e}\eta_i\Big)\Big]^-\\
\nonumber&\qquad+{\e}\sum_{i=1}^{x/\e}h(\eta_i|p)+\e I^F_c(y)\Big\}.
 \end{align}
 For $s\in (1,\infty)$ set $y^*=0$ while for $s=\infty$ set $y^*\leq 0$
 so that $I^F(y^*)<\infty$.
 The infimum in \eqref{IUfinal}
can be bounded above by substituting $y=y^*\leq 0$, hence 
\begin{align}
\nonumber I_u^*(x,c)&\leq\inf \limits_{\{\eta\}^\e_{x} \in M_1^s(\Sigma_p)} \Big\{ \Big[ \e\Big(y^* - F_4(\{\eta\}^\e_{x})-\log C_{\kappa}\Big)^+-I^{F}\Big(x,\frac{\e}{x} \sum \limits_{i=1}^{x/\e}\eta_i\Big)\Big]^-\\
\nonumber&\qquad+{\e}\sum_{i=1}^{x/\e}h(\eta_i|p)+\e I^F_c(y^*)\Big\}\\
\label{IUbound}&=\inf \limits_{\{\eta\}^\e_{x} \in M_1^s(\Sigma_p)} \Big\{ I^{F}\Big(x,\frac{\e}{x} \sum \limits_{i=1}^{x/\e}\eta_i\Big)+{\e}\sum_{i=1}^{x/\e}h(\eta_i|p)\Big\}+\e I^F_c(y^*).
\end{align}
So with \eqref{ILbound} and \eqref{IUbound} applied to \eqref{chiIL} and \eqref{chiIU}, we obtain
\begin{align}
\label{chi2bounds}-\e I^F_c(y^*)\leq \inf \limits_{\{\eta\}^\e_{x} \in M_1^s(\Sigma_p)} \Big\{ I^{F}\Big(x,\frac{\e}{x} \sum \limits_{i=1}^{x/\e}\eta_i\Big)+\e\sum_{i=1}^{x/\e}h(\eta_i|p)\Big\}+\lim \limits_{n\to \infty}& \frac{1}{k}\log \chi\leq \e \log C_{\kappa}.
\end{align}
To conclude the proof we show that 
\begin{equation}
\label{INF-1}\inf \limits_{\{\eta\}^\e_{x} \in M_1^s(\Sigma_p)} \Big\{ I^{F}\Big(x,\frac{\e}{x} \sum \limits_{i=1}^{x/\e}\eta_i\Big)+\e\sum_{i=1}^{x/\e}h(\eta_i|p)\Big\}
\end{equation}
does not depend on $\e$ so we can take the limit $\e\to 0$ in \eqref{chi2bounds}. First, since specific relative entropy is affine, we have
\begin{equation}
{\e}\sum_{i=1}^{x/\e}h(\eta_i|p)=xh\Big(\frac{\e}{x} \sum_{i=1}^{x/\e}\eta_i\mid p\Big),
\end{equation}
so \eqref{INF-1} becomes 
\begin{equation}
\label{INF-2}\inf \limits_{\{\eta\}^\e_{x} \in M_1^s(\Sigma_p)} \Big\{ I^{F}\Big(x,\frac{\e}{x} \sum \limits_{i=1}^{x/\e}\eta_i\Big)+xh\Big(\frac{\e}{x} \sum_{i=1}^{x/\e}\eta_i\mid p\Big)\Big\}.
\end{equation}
Now observe that $\frac{\e}{x} \sum_{i=1}^{x/\e}\eta_i \in M_1^s(\Sigma_p)$ and thus
\begin{align}
\label{lowerE}&\inf \limits_{\{\eta\}^\e_{x} \in M_1^s(\Sigma_p)} \Big\{ I^{F}\Big(x,\frac{\e}{x} \sum \limits_{i=1}^{x/\e}\eta_i\Big)+xh\Big(\frac{\e}{x} \sum_{i=1}^{x/\e}\eta_i\mid p\Big)\Big\} \\
&\qquad \geq \inf \limits_{\eta \in M_1^s(\Sigma_p)} \Big\{ I^{F}\big(x,\eta \big)+xh\big(\eta\mid p\big)\Big\}.\nonumber
\end{align}
At the same time when considering the infimum over the vectors $\{\eta\}^\e_{x}$ we could restrict ourselves to having all the measures being equal, i.e. $\eta_i=\eta$ for $i=1,\ldots,x/\e$, and hence the reverse inequality also holds, therefore
\begin{align}
\label{equalE} &\inf \limits_{\{\eta\}^\e_{x} \in M_1^s(\Sigma_p)} \Big\{ I^{F}\Big(x,\frac{\e}{x} \sum \limits_{i=1}^{x/\e}\eta_i\Big)+xh\Big(\frac{\e}{x} \sum_{i=1}^{x/\e}\eta_i\mid p\Big)\Big\} \\
&\qquad= \inf \limits_{\eta \in M_1^s(\Sigma_p)} \Big\{ I^{F}\big(x,\eta \big)+xh\big(\eta\mid p\big)\Big\}.
\nonumber
\end{align}
Finally, observe that $I^F_c(y^*)$ is finite: if $s=\infty$ we chose $y^*$ that
way, while  $s\in (1,\infty)$ implies that
there are
positive and negative drifts in the support of $p$, 
and hence $0$ is in the domain of $I^F_c$. With this, substituting  
\eqref{equalE} in \eqref{chi2bounds}, we take $\e\to 0$ to conclude the proof. 
\end{proof}

\section{Appendix A: LDP for the conditional random walk on random environment}
In this short appendix, we show that $I^F$ in \eqref{eq-IL} has, for ergodic laws
on the environment, a natural interpretation in terms of the rate function
for the large deviations of hitting times in a conditioned environment. 

 \begin{prop}\label{appendixA}  Fix   $\eta\in M_1^e(\Sigma)$. Then, 
\begin{align*}
\lim \limits_{k\to \infty}-\frac{1}{k}\log \pr^{\hat{\omega}}_1[ \tau_{xk}\leq k]=\lim \limits_{k\to \infty}I^F(x,R_{xk})=I^F(x,\eta), \quad 
    \eta-a.s..
\end{align*}

\end{prop}

\begin{proof}
In Lemma~\ref{lemmaPhat} and Lemma~\ref{lemmaapproxPhat} we were able to approximate $\pr^{\hat\omega}_1[\tau_{xk}\leq k]$ using $\hat{I}^L$ and error terms. Consider $J=J(k)$ such that $\lim_{k\to \infty} J(k)=\infty$, but $J=o(k)$. In this proof we show that 
\begin{enumerate}
\item  $\lim \limits_{L\to \infty} \lim \limits_{k\to \infty}\hat I_J^L(x,k,\omega)= I^F(x,\eta)$, $\eta$ almost surely.
    \item 
       $ 
\lim \limits_{L\to \infty} \lim \limits_{k\to \infty} \left|-\frac{1}{k}\log \pr^{\hat\omega}_1[\tau_{xk}\leq k]-\hat I_J^L(x,k,\omega) \right|=0,$
$\eta$  almost surely.
\end{enumerate}

For the first statement observe that, see \eqref{Iphidef},
\begin{equation}
 \label{IL1}  \hat{I}_0^L(x,k,\omega)=I^{\phi}_{L}(x,R_{xk}). 
\end{equation}

Moreover by Lemma~\ref{Jto1},
\[
\lim \limits_{k\to \infty}|\hat{I}_0^L(x,k,\omega)-\hat{I}_J^L(x,k,\omega) |=0,\qquad\eta\mbox{-a.s. }
\]
Since  $I^{\phi}_{L}(x,R_{xk})\to I^{\phi}_{L}(x,\eta)$ almost surely as $k\to \infty$, the statement is proved. 

For the second statement, according to Lemma~\ref{lemmaapproxPhat} we have 
\begin{align*}
    &
    |\hat{I}_J(x,k,\omega)-\hat{I}^L_J(x,k,\omega)|\\
    &\leq  \frac{\alpha_6}{k}\sum \limits_{i=J}^{xk-1}\pr^{\hat{\omega}^L}_i[\tau_{i+1}\geq M]+ \frac{\alpha_6}{k}\sum \limits_{i=J}^{xk-1}\pr^{\hat{\omega}}_i[\tau_{i+1}\geq M]
    +\frac{\alpha_7}{k}\sum\limits_{i=J}^{xk-1}H^\omega_i(L,M)\\
    &\leq  \frac{2\alpha_6}{k}\sum \limits_{i=J}^{xk-1}\pr^{\hat{\omega}}_i[\tau_{i+1}\geq M]+\frac{\alpha_7}{k}\sum\limits_{i=J}^{xk-1}H^\omega_i(L,M). 
\end{align*}
We prove now that $\eta$ almost surely, for every $M$,
\begin{equation}
   \label{limitH} \lim \limits_{L \to \infty} \lim \limits_{k\to \infty} \frac{1}{k}\sum\limits_{i=J}^{xk-1}H^\omega_i(L,M) =0
\end{equation}
  Indeed, 
  \begin{align}
      \nonumber\frac{1}{k}\sum\limits_{i=J}^{xk-1}H^\omega_i(L,M)&=\frac{M}{k}\sum\limits_{i=J}^{xk-1} \sum \limits_{j=i-M}^{i+M}\xi_{j}(L-1,\omega)\\
      \nonumber&\leq\frac{M}{k}\sum\limits_{i=J}^{xk-1} \sum \limits_{j=i-M}^{i+M}\bar\xi(L-1,\theta^{j-M-L+1}\omega)\\
      \label{Hbound}&=\frac{M}{k}\sum \limits_{j=-M}^{M}\sum\limits_{i=J}^{xk-1} \bar\xi(L-1,\theta^{i+j-M-L+1}\omega)\\
      \nonumber&\to Mx\int \sum \limits_{j=-M}^{M}\bar\xi(L-1,\omega) \eta(d\omega), \quad \mbox { as }k\to \infty,
  \end{align}
  where the limit is due to the ergodic theorem.
 Lemma~\ref{lemmasigma} then concludes the proof of \eqref{limitH}. 
Recalling that $J=o(k)$,  the proof of Proposition~\ref{appendixA} is completed using \eqref{limitH} and Lemma~\ref{lemmaptau}. 
\end{proof}

\bibliography{main}

\begin{thebibliography}{1}

\bibitem{LDP}
F.~Comets, N.~Gantert, and O.~Zeitouni.
\newblock Quenched, annealed and functional large deviations for
  one-dimensional random walk in random environment.
\newblock {\em Probability theory and related fields}, 118(1):65--114, 2000.

\bibitem{popovcomets}
F.~Comets, M.~Menshikov, and S.~Popov.
\newblock Lyapunov functions for random walks and strings in random
  environment.
\newblock {\em The Annals of Probability}, 26(4):1433--1445, 1998.

\bibitem{DZ}
A.~Dembo and O.~Zeitouni.
\newblock {\em Large deviation techniques and applications}.
\newblock Springer, 1998.
\newblock 2nd edition.

\bibitem{solomon}
F.~Solomon.
\newblock Random walks in a random environment.
\newblock {\em The annals of probability}, 3(1):1--31, 1975.

\bibitem{coursezeitouni}
O.~Zeitouni.
\newblock Random walks in random environment.
\newblock {\em Lecture notes in Mathematics}, 1837:190--312, 2004.

\end{thebibliography}

\end{document}